\documentclass[12pt]{amsart}
\usepackage{amsthm,amsopn,amsfonts,amssymb,pb-diagram}
\usepackage{tabmac}

\usepackage{euler, amsfonts, amssymb, latexsym, epsfig, epic}
\usepackage{epstopdf}

\font\co=lcircle10
\def\boxcross{\ \smash{\lower6.5pt\hbox{\rlap{\hskip4.5pt\vrule height13.5pt}}
                \raise0pt\hbox{\rlap{\hskip-2pt \vrule height.4pt depth0pt
                        width13.5pt}}}\hskip12.7pt}

\def\boxelbow{\ \hskip.1pt\smash{%
               \hbox{\co \hskip 5.5pt\rlap{\mathsurround=0pt\rlap{\mathsurround=0pt\char'006}\lower0.4pt\rlap{\char'004}}
                \lower6.5pt\rlap{\hskip-0.2pt\vrule height3pt}
                \raise3.5pt\rlap{\hskip-0.2pt\vrule height3.2pt}}
                \hbox{%
                  \rlap{\hskip-6.4pt \vrule height.4pt depth0pt
width2.5pt}%
                  \rlap{\hskip4.05pt \vrule height.4pt depth0pt
width3.1pt}}}
                \hskip8.7pt}

\usepackage[backref]{hyperref}
   \renewcommand*{\backrefalt}[4]{%
     \ifcase #1 %
     \or
     \vspace{0pt} \footnotesize (In \S#2)%
     \else
     \vspace{0pt} \footnotesize (In \S#2)
     \fi}
\renewcommand*{\backref}[1]{}

\newtheorem{theorem}{Theorem}[section]
\newtheorem{thm}[theorem]{Theorem}
\newtheorem*{thm*}{Theorem}
\newtheorem{lemma}[theorem]{Lemma}
\newtheorem{lem}[theorem]{Lemma}

\newtheorem{prop}[theorem]{Proposition}

\newtheorem{cor}[theorem]{Corollary}

\theoremstyle{definition}

\theoremstyle{remark}

\newtheorem{Example}[theorem]{Example}
\newtheorem{remark}[theorem]{Remark}
\newcommand{\Zo}{\mathring{Z}}
\newcommand{\Yo}{\mathring{Y}}
\newcommand{\Xo}{\mathring{X}}
\newcommand{\Pio}{\mathring{\Pi}}
\newcommand{\SR}{\mathrm{SR}}
\newcommand{\In}{\mathrm{In}}
\def\a{{\mathbf{a}}}
\def\b{{\mathbf{b}}}

\newcommand{\iso}{\cong}

\newcommand{\I}{{\mathcal I}}

\newcommand{\Z}{\mathbb Z}
\newcommand{\onto}{\twoheadrightarrow}
\newcommand{\into}{\hookrightarrow}

\newcommand{\PP}{\mathbb{P}}
\newcommand{\Sp}{\mathrm{Sp}}

\newcommand\defn[1]{{\bf #1}}
\newcommand\newword[1]{{\bf #1}}

\DeclareMathOperator{\Spec}{Spec}

\DeclareMathOperator{\codim}{codim}
\DeclareMathOperator{\lex}{lex}

\newcommand{\Fl}{F \ell}
\renewcommand\O{{\mathcal O}}

\newcommand\junk[1]{}

\begin{document}

\title{Projections of Richardson Varieties}
\author{Allen Knutson}
\address{Department of Mathematics, Cornell University, Ithaca, NY 14853 USA}
\email{allenk@math.cornell.edu}
\thanks{AK was partially supported by NSF grant DMS-0604708.}
\author{Thomas Lam}
\address{Department of Mathematics, University of Michigan, Ann Arbor, MI
  48109 USA}
\email{tfylam@umich.edu}
 \thanks{TL was supported by NSF grants DMS-0652641 and DMS-0901111, and by a Sloan Fellowship.}
\author{David E Speyer}
\address{Department of Mathematics, University of Michigan, Ann Arbor, MI
  48109 USA}
\email{speyer@umich.edu}
\thanks{DES was supported by a Research Fellowship from the Clay Mathematics Institute}
\date{\today}

\begin{abstract}
  While the projections of Schubert varieties in a full generalized
  flag manifold $G/B$ to a partial flag manifold $G/P$ are again
  Schubert varieties, the projections of Richardson varieties
  (intersections of Schubert varieties with opposite Schubert varieties)
  are not always Richardson varieties. The stratification of $G/P$
  by projections of Richardson varieties arises in the theory of
  total positivity and also from Poisson and noncommutative geometry.

  In this paper we show that many of the geometric
  properties of Richardson varieties hold more generally for projected
  Richardson varieties; they are normal, Cohen-Macaulay, have
  rational resolutions, and are compatibly Frobenius split with
  respect to the standard splitting. Indeed, we show that the
  projected Richardson varieties are the {\em only} compatibly
  split subvarieties, providing an example of the recent theorem
  [Schwede, Kumar-Mehta] that a Frobenius split scheme has only
  finitely many compatibly split subvarieties. (The $G/B$ case was
  treated by [Hague], whose proof we simplify somewhat.)
  
  One combinatorial analogue of a Richardson variety is the order
  complex of the corresponding Bruhat interval in $W$;
  this complex is known to be an EL-shellable ball [Bj\"orner-Wachs '82]. 
  We prove that the projection of such a complex into the order complex 
  of the Bruhat order on $W/W_P$ is again a shellable ball.
  This requires extensive analysis of ``$P$-Bruhat order'',
  a generalization of the $k$-Bruhat order of [Bergeron-Sottile '98].
  In the case that $G/P$ is minuscule (e.g. a Grassmannian), 
  we show that its Gr\"obner degeneration takes each projected Richardson 
  variety to the Stanley-Reisner scheme of its corresponding ball.

\junk{
  We mention a number of special features of the Grassmannian case,
  which will be treated in a separate paper.}
\end{abstract}

\maketitle

\setcounter{tocdepth}{1}
{\footnotesize \tableofcontents}

\section{Introduction, and statement of results}\label{sec:intro}

Fix a reductive algebraic group $G$ over an algebraically closed field of arbitrary characteristic, with upper and lower Borel subgroups $B = B^+$ and $B^-$.
Fix a parabolic subgroup $P \supseteq B^+$, with $W$ and $W_P$ the corresponding Coxeter groups.  For $u$ in $W/W_P$, the \newword{Schubert cell} $\Yo_{u}$ is $B^- u P/P$ and 
$Y_{u}$ is the closure of $\Yo_{u}$. 
We also define the opposite Schubert cells:
given $w \in W/W_P$, the \newword{opposite Schubert cell} $\Yo^w$ is 
$B^+ w P/P$ and $Y^w$ is the closure of $\Yo^w$.
Both are affine spaces, of codimension and dimension $\ell(w^P)$ respectively,
where $\ell:W\to \Z$ is the length function on $W$ and $w^P$ is the
shortest representative of $w$ in $W/W_P$. 

\junk{
The cell $\Yo^w$ is an affine space of dimension $\ell(w^P)$ and
$Y^w = \bigcup_{v \leq w} \Yo^w$.
The Schubert cell $\Yo_{u}$ is an affine space of dimension $\ell((w_0)^P) - \ell(u^P)$ and
$Y_{u} = \bigcup_{v \geq u} \Yo_{v}$.
}

The \newword{open Richardson variety} $\Yo_u^w$ is $\Yo_u \cap \Yo^w$. 
This is nonempty if and only if $u^P \leq w^P$, 
and it has dimension $\ell(w^P) - \ell(u^P)$.
It depends only on the classes of $u$ and $w$ in $W/W_P$.
The \newword{Richardson variety} $Y_u^w$ is the closure of $\Yo_u^w$, 
and we have $Y_u^w = \coprod_{u \leq x \leq y \leq w} \Yo_x^y$. 

When $P=B^+$, so that $W_P = \{ e \}$, we write $\Xo$ and $X$ instead
of $\Yo$ and $Y$.

We write $\pi$ for the projection $G/B \to G/P$, and also for the map
$W \to W/W_P$.  When necessary, we will write $\pi_P$ to indicate
the parabolic $P$.

While $\pi(Y_w)$ and $\pi(Y^w)$ are again Schubert and opposite Schubert
varieties, $\pi(Y_u^w)$ may not again be a Richardson variety.  These
``projected Richardson varieties'' $\pi(Y_u^w)$ were previously studied in \cite{Lus,Rie,GY}.
In this paper we show that all the standard (and some less well-known)
properties of Richardson varieties hold for the $\pi(Y_u^w)$:

\begin{thm*}\
  \begin{enumerate}
  \item Projected Richardson varieties are normal and Cohen-Macaulay, 
    and have rational resolutions (this definition is recalled in
    \S \ref{sec:frobsplit}).
  \item Under the standard Frobenius splitting on $G/P$, the 
    projected Richardson varieties are exactly 
    the compatibly split subvarieties. (In the $P=B$ case this
    was shown recently in \cite{Hague}.)
  \item The projection to $W/W_P$ of the order complex of a Bruhat
    interval in $W$ is a shellable ball.
  \item If $G/P$ is minuscule (definition recalled below),
    then under the standard Gr\"obner degeneration of $G/P$ to a
    Stanley-Reisner scheme, each projected Richardson variety 
    degenerates to the Stanley-Reisner scheme of its corresponding ball.
  \end{enumerate}
\end{thm*}

We could not find a reference for the statement that usual Richardson varieties have rational 
resolutions in all characteristics, and this is established by a separate argument in Appendix \ref{sec:appendix}.

Recall that a \defn{minuscule} $G/P$ is the closed $G$-orbit in the
projectivization of a minuscule irrep $V_\omega$, meaning one
whose only weights are the extremal weights $W\cdot \lambda$.
This forces $P$ to be maximal, and in the $G=GL_n$ case,
all such $G/P$ (the Grassmannians) are minuscule, where
the $V_\omega$ are the Pl\"ucker embedding spaces.

We believe that for nonminuscule embeddings, the Gr\"obner degeneration
should be replaced by the Chiriv\`i-Lakshmibai-Seshadri-Littelmann
degeneration \cite{Ch}, and hope to address this in a separate paper.

We give an example at the end of \S \ref{sec:frobsplit} of a property
of Richardson varieties not shared by general projected Richardson varieties.

After this work was completed, we found the preprint \cite{BC} which
has significant overlap with this paper. They prove, in characteristic zero, that 
projected Richardson varieties (and others) are normal, C-M, and
have rational singularities, and unlike us, make a study of their
singular loci. To prove the relevant cohomology vanishing statements,
instead of our rather specific appeals to \cite{LL} they prove
some more general statements about Mori contractions.
They do not consider the degenerations we do, 
and so are not faced with combinatorial questions about ``$P$-Bruhat order''.

\medskip

{\bf Acknowledgements.} We thank Michel Brion, Shrawan Kumar, and
the referee for helpful comments and suggestions.

\section{Bruhat intervals and the $P$-Bruhat order}\label{sec:Bruhat}
Let $W$ be a Coxeter group and $\{s_i \mid i \in I\}$ be its set of simple generators.  We let $\ell:W \to \Z$ denote the length function of $W$.  The \newword{left weak order} of $W$ is defined by $w \prec v$ if there exists $u \in W$ such that $uw = v$ and $\ell(u)+\ell(w) = \ell(v)$.  We let $<$ denote the {Bruhat order} of $W$, and let $\lessdot$ denote a cover in Bruhat order.   The \newword{(right) descent set} of $w \in W$ is $\{i \in I \mid ws_i < w\}$.

Let $W_P \subset W$ be a parabolic subgroup.  We write $W^P$ for the minimal length coset representatives of $W/W_P$.  When $W_P$ is finite (as it will be), we also let $W^P_{\max}
$ denote the maximal length coset representatives of $W/W_P$.  Every $w \in W$ has a unique factorization as $w = w^Pw_P$ where $w^P \in W^P$, $w_P \in W_P$ and $\ell(w) = \ell(w^P)+\ell(w_P)$.  We call this the \newword{parabolic factorization} of $w$.  We denote by $\pi: W \to W/W_P$ the projection, and if necessary we write $\pi_P$ to indicate the parabolic.
We will occasionally identify $W/W_P$ with $W^P$.

Let $w,v \in W$.  We say that $w$ \newword{$P$-covers} $v$, denoted $w
\gtrdot_P v$, if $w \gtrdot v$ and $wW_P \neq vW_P$.  Let $\leq_P$
denote the \newword{$P$-Bruhat order}, the transitive closure of $P$-covers.
Thus $u \leq_P w$ if there is a saturated chain $u = v_0 \lessdot v_1
\lessdot \cdots \lessdot v_l = w$ in $W$ such that $\pi_P(v_0) <
\pi_P(v_1) < \cdots < \pi_P(v_l)$.  If $W=S_n$ and $W_P = S_k \times
S_{n-k}$, this is the relation $\leq_k$ studied in \cite{BS}.  We
write $[v,w]_P$ for a $P$-Bruhat interval in the $P$-Bruhat order.

The following combinatorial result will be established in \S \ref{sec:shelling}.

\begin{prop}\label{P:uniquelift}
Let $x \in W$ and $C$ be a coset of $W_P$.  Let $C_{\geq x}$ be the set of elements in $C$ greater than $x$.  Then $C_{\geq x}$ is either empty, or contains a unique minimum $z$.  In the latter case we have $x \leq_P z$.
\end{prop}

\begin{lemma}\label{L:Pleftweak}
Suppose $v \lessdot_P w$.  Let $v = v^P v_P$ and $w = w^P w_P$ be parabolic factorizations where $v^P, w^P \in W^P$ and $v_P,w_P \in W_P$.  Then $w_P \preceq v_P$ in left weak order.  In particular, the descent set of $v_P$ contains the descent set of $w_P$.
\end{lemma}
\begin{proof}
Since $v \lessdot w$, a reduced word for $v$ can be obtained from removing a simple generator from a reduced word for $w$.  Take a reduced word $\a \b$ for $w$ where $\a$ is a reduced word for $w^P$ and $\b$ is a reduced word for $w_P$.  The removed simple generator is inside $\a$, for otherwise $w^P = v^P$.  But then it follows that $v_P$ has a reduced word of the form $\a' \b$, so that $w_P \preceq v_P$.
\end{proof}

\begin{lemma}\label{L:lengthadd}
Suppose $u \leq v$ and $x \leq y$ in $W$, and $x = uz$ and $y=vz$ are
both length-additive factorizations with $z \in W_P$.  Then $x \leq_P y$ if and only if $u \leq_P v$.
\end{lemma}
\begin{proof}
It suffices to prove the case that $z = s_i \in W_P$.  Suppose $u \leq_P v$.  Let $u = w_0 \lessdot_P w_1 \lessdot_P w_2 \lessdot_P \cdots \lessdot_P w_N= v$ be a saturated chain.  
Since $u s_i$ is length-additive, the reflection $s_i$ is not a right descent of $u$.
So, by Lemma~\ref{L:Pleftweak}, $s_i$ is also not a right descent of $w_1$, and $w_1 s_i$ is length-additive.
Continuing in this manner, we see that $\ell(w_j s_i) = \ell(w_j) + 1$ for every $j$. 
So we have the chain of covers $x \lessdot w_1s_i \lessdot w_2s_i \lessdot \cdots \lessdot  y$.
Moreover, $(w_j s_i) (w_{j+1} s_i)^{-1} = w_j w_{j+1}^{-1}$ so, since $ w_j w_{j+1}^{-1}$  is assumed not to be in $W_P$, we have $(w_j s_i) (w_{j+1} s_i)^{-1} \not \in W_P$ as well.

So we have a chain of $P$-Bruhat covers $x \lessdot_P w_1s_i \lessdot_P w_2s_i \lessdot_P \cdots \lessdot_P  y$, and we deduce that $x \leq_P y$.
The reverse direction is similar, starting with the fact that $s_i$ is a right descent of $y$ and working down the chain. 
\end{proof}

Define an equivalence relation on the set of $P$-Bruhat intervals, generated by the relations $[u,v]_P \sim [x,y]_P$ if there is $z \in W_P$ such that $x=uz$ and $y=vz$ are both length-additive.  We let $\langle u,v\rangle_P$ denote an equivalence class of $P$-Bruhat intervals.

\begin{lemma} \label{L:Shift}
Suppose that $u \leq_P w$.  Then $u (w_P)^{-1} \leq_P w^P$. 
In particular, every equivalence class $\langle u, v \rangle_P$ has a representative $[x,y]_P$ with $y \in W^P$.
\end{lemma}
\begin{proof}
By Lemma \ref{L:Pleftweak}, we have $w_P \preceq u_P$, so that $u_P = zw_P$ for some $z \in W_P$.  The result then follows from Lemma \ref{L:lengthadd}.
\end{proof}

\begin{prop} \label{P:TopBottom}
If $u \leq w$ and $w \in W^P$, then $u \leq_P w$.
\end{prop}

\begin{proof}
Let $C$ be the coset $w W_P$.  Since $w$ is the minimum of $C_{\geq u}$, the result follows immediately from Proposition \ref{P:uniquelift}.
\end{proof}

Let $Q(W,W_P)$ be the set of equivalence classes of $P$-Bruhat intervals.
By Proposition~\ref{P:TopBottom} and Lemma~\ref{L:Shift}, each element of $Q(W,W_P)$ can be represented uniquely by a pair $(u,w)$ where $w\in W^P$. 

Suppose now that $W$ is a finite Weyl group.  Writing $u$ as $u' x$, where $u' \in W^P_{\max}$ and $x \in W_P$, we see that $Q(W,W_P)$ is in bijection with the set of triples $(u', w, x)$ where $w \in W^P$, $u'\in W^P_{\max}$, $x \in W_P$, and $u'x \leq w$. 
Using this last description, we see that our set $Q(W, W_P)$ is the same as Rietsch's $Q^J$ \cite{Rie}.

\section{Projected Richardson varieties} \label{sec:proj}

We now introduce the projected Richardson varieties, our principal
objects of study. Fix a parabolic $P \supset B$ in $G$. 

Let $u \leq_P w$ be a $P$-Bruhat interval in $W$.  We define $\Pio_u^w = \pi(\Xo_u^w)$ and $\Pi_u^w = \pi(X_u^w)$, the \newword{open} and \newword{closed projected Richardson varieties}.  Here $\Pi$ stands for ``projection'', ``positroid'' and ``Postnikov''.  The Grassmannian case of projected Richardson varieties, called \newword{positroid varieties}, are studied in \cite{KLS} motivated by work of Postnikov \cite{Pos}.  Note that, since $X_u^w$ is proper and irreducible, $\Pi_u^w$ is likewise, and is therefore the closure of $\Pio_u^w$. 

The projected Richardson varieties were studied previously by Lusztig \cite{Lus} and Rietsch \cite{Rie} in the context of total positivity, and by Goodearl and Yakimov \cite{GY} in the context of Poisson geometry. Very recently (while we were finishing this paper), 
they were studied in \cite{BC}. 

We now discuss the elementary geometry and combinatorics of the map $\pi$ from $X_u^w$ to $\Pi_u^w$.  The next lemma relates projected Richardson varieties to the set $Q(W, W_P)$ of \S \ref{sec:Bruhat}:

\begin{lemma} \label{L:WhyQ}
Suppose $u \leq_P w$ and $u' \leq_P w'$. If $(u,w) \sim (u', w')$ in $Q(W, W_P)$, then $\pi(\Xo_u^w) = \pi(\Xo_{u'}^{w'})$. 
Moreover, $\pi$ is injective on $\Xo_u^w$.
\end{lemma}

\begin{proof}
%

By Lemma \ref{L:lengthadd}, it suffices to consider the case where $w' = ws$ and $u'=us$ are length-additive, and $s$ is a simple reflection in $W_P$.

Let $R$ be the parabolic subgroup whose dimension is one more than $B$ and which corresponds to $s$.
So $R \subseteq P$. Since $\pi_P$ factors through $\pi_R$, it is enough to show that $\pi_R(\Xo_{u}^w) = \pi_R(\Xo_{u'}^{w'})$.

The map $\pi_R$ is a $\PP^1$-bundle. 
Let's focus on a single fiber $F$. The Schubert stratification of $G/B$ divides this fiber into a point $p$ and an affine line, and the opposite Schubert stratification marks off another point $q$. (Generically, $p \neq q$, but in some fibers they conicide.)
The condition that $u < us$ ensures that either
\begin{itemize}
\item[(u1)] the intersection $\Xo_u \cap F$ is $F \setminus \{ p\}$ and $\Xo_{us} \cap F$ is $\{ p \}$ or
\item[(u2)] both intersections are empty.
\end{itemize}
Similarly, since $w < ws$, either
\begin{itemize}
\item [(w1)] the intersection $\Xo^w \cap F$ is $\{ q \}$ and
  $\Xo^{ws} \cap F$ is $F \setminus \{ q \}$ or
\item [(w2)] both intersections are empty.
\end{itemize}
If either (u2) or (w2) holds, then $F \cap \Xo_u^w = F \cap \Xo_{us}^{ws} = \emptyset$. 
If both (u1) and (w1) hold then, if $p \neq q$, the intersections $F \cap \Xo_u^w$ and $F \cap \Xo_{us}^{ws}$ are each a single point; if $p=q$, these intersections are both empty.
In either case, we see that $\pi_R(F)$ is in $\pi_R(\Xo_u^w)$ if and only if it is in $\pi_R(\Xo_{us}^{ws})$ and that, if it is, the fiber above it is a single point in both cases.
\end{proof}

It is known \cite{Lus,Rie} that such $\Pio_u^w$ are smooth, since the projection is an isomorphism with an open Richardson variety which is known to be smooth.

\begin{cor} \label{C:ProjectionDimension}
For $u \leq_P w$, the dimension of $\Pi_u^w$ is $\ell(w) - \ell(u)$.  
The \newword{interior} $\Pio_u^w$ is smooth.
\end{cor}

So, given an equivalence class $[(u,w)]$ in $Q(W, W_P)$, the open
projected Richardson variety $\Pio_u^w$ is well defined.

A variation of the above proof lets us describe $\pi(X_u^w)$ for any Richardson variety $X_u^w$: recall that \cite{KM} the \defn{Demazure product} $\circ$ of $W$
is defined, for $s$ a simple reflection, by
$$
w \circ s  = \begin{cases} ws & \mbox{if $ws > w$,} \\
w &\mbox{otherwise.}
\end{cases}
$$
One then defines $w \circ v$ by picking a reduced expression $v =
s_{1} s_{2} \cdots s_{\ell}$ and setting $w \circ v = (((w
\circ s_{1}) \circ s_{2}) \circ \cdots) \circ s_{\ell}$. (So
if $wv$ is length-additive, then $w\circ v = wv$.) The result does
not depend on the choice of reduced expression.

\begin{prop} \label{P:EquivalenceClassClosed}
Let $u \leq w$ and let $x$ be the element of $W_P$ such that
 $ux$ is $W_P$-maximal. Then $\pi(X_u^w) = \pi(X_{ux}^{w \circ x})$.
 \end{prop}

\begin{proof}
Our proof is by induction on the length of $x$; if $x=e$ then the
claim is trivial. If $x \neq e$ then $\ell(s_i x) <\ell(x)$ for some
$i \neq k$. We will show that $\pi(X_{u}^w) = \pi(X_{u s_i}^{w \circ
s_i})$, at which point we are done by induction. Define $R$ as in the proof of Lemma~\ref{L:WhyQ}. Once
again, it is enough to show that $\pi_R(X_u^w)= \pi_R(X_{u s_i}^{w
\circ s_i})$. Let $U = \pi_R(X_{u}) = \pi_R(X_{u s_i})$ and $V =
\pi_R(X^{w}) = \pi_R(X^{w s_i})$. Clearly, both $\pi_R(X_u^w)$ and $
\pi_R(X_{u s_i}^{w \circ s_i})$ lie in $U \cap V$. Let $z$ be a point
of $U \cap V$, so $\pi_R^{-1}(z) \cong \PP^1$. We will look at the
intersection of $\pi_R^{-1}(z)$ with $X_u$, $X_{u s_i}$, $X^w$ and
$X^{w s_i}$.

The pair $(X_u \cap \pi_R^{-1}(z), X_{u s_i} \cap \pi_R^{-1}(z))$ is
either $(\emptyset, \emptyset)$, $(\PP^1, \{ \mbox{pt} \})$ or
$(\PP^1, \PP^1)$. If $\ell(w s_i) > \ell(w)$ then $w \circ s_i = w
s_i$. In that case, the pair  $(X^{w s_i} \cap \pi_R^{-1}(z), X^w
\cap \pi_R^{-1}(z))$ is also limited to one of the three preceding
cases. If, on the other hand, $\ell(w s_i)<\ell(w)$ then $w \circ
s_i=w$ and $X^w \cap \phi^{-1}(z)$ is either $\emptyset$ or $\PP^1$.
Checking all $15$ cases, we see that, in each case, $\pi_R^{-1}(z)
\cap X_u^w$ is nonempty if and only if $\pi_R^{-1}(z) \cap X_{u
s_i}^{w \circ s_i}$ is.
\end{proof}

\begin{cor} \label{C:NotInj}
If $\pi_P$ is birational on $X_u^w$, then $u \leq_P w$.
\end{cor}

\begin{proof}
If $u \not\leq_P w$ then Proposition~\ref{P:EquivalenceClassClosed} shows that $\pi(X_u^w)$ can be parametrized by $X_{u'}^{w'}$ where $\dim X_{u'}^{w'} < \dim X_u^w$. 
\end{proof}

So projecting Richardson varieties for non $P$-Bruhat intervals does not produce any new closed projected Richardson varieties.
Given a projected Richardson variety $\Pi$, we will call $X_u^w$ a \newword{Richardson model} for $\Pi$ if $\pi(X_u^w)=\Pi$ and $u \leq_P w$. 

\begin{remark} The same is not true for projected open Richardson varieties: $\pi(\Xo_u^w)$ is not always equal to $\pi(\Xo_{u'}^{w'})$ for some $u' \leq_P w'$. Consider the situation of the flag variety $\Fl_3$ of $SL(3)$ projecting to the Grassmannian $G(1,3) \cong \PP^2$. The image of the big Richardson cell is all of $\PP^2$ except for the two points $(1:0:0)$ and $(0:0:1)$; whereas $\Xo_{s_1}^{s_1 s_2 s_1}$ and $\Xo_{e}^{s_1 s_2}$ both map to the complement of the three boundary lines.
\end{remark}

For our purposes, Rietsch's closure result \cite{Rie} can be formulated as follows.

\begin{prop} \label{P:Rietsch}
Every point of $G/P$ is in $\Pio_u^w$ for a unique $(u,w)$ in $Q(W, W_P)$. 
If $u \leq_P w$, then $\Pi_u^w = \coprod_{u \leq u' \leq_P w' \leq w} \Pio_{u'}^{w'}$.
\end{prop}
\begin{proof}
The first statement, together with the second statement for $w \in W^P$, can be found in \cite{Rie}.  Now let $u \leq_P w$ be arbitrary.  If $u \leq u' \leq_P w' \leq w$, then $\Xo^{w'}_{u'} \subset X^w_u$, which gives the inclusion $\Pi_u^w \supseteq \bigcup_{u \leq u' \leq_P w' \leq w} \Pio_{u'}^{w'}$.  

For the other inclusion, using Rietsch's statement for $w \in W^P$ and Lemmata \ref{L:lengthadd}, \ref{L:WhyQ}, and \ref{L:Shift}, we may assume that the result holds for $x \leq_P y$ and prove it for $u \leq_P w$, where $u = xs$ and $w = ys$ are length-additive.  Suppose $x \leq x' \leq_P y' \leq y$.  If $x's > x'$, then $y's >y'$ as well by Lemma \ref{L:Pleftweak}, and we set $u' =x's$ and $w' = y's$.  If $x's < x'$, then by \cite[Proposition 5.9]{Hum} we must have $xs \leq x'$, and we set $u' = x'$ and $w' = y'$.  In both cases $u \leq u' \leq_P w' \leq w$ and $\Pio_{u'}^{w'} = \Pio_{x'}^{y'}$ by Lemma \ref{L:WhyQ}.  Since we have assumed that $\Pi_x^y \subseteq \bigcup_{x \leq x' \leq_P y' \leq y} \Pio_{x'}^{y'}$ it follows that
$\Pi_u^w \subseteq \bigcup_{u \leq u' \leq_P w' \leq w} \Pio_{u'}^{w'}$.
\end{proof}

So the $\Pi_u^w$, where $(u,w)$ ranges through $Q(W, W_P)$, form a stratification of $G/P$.  Proposition \ref{P:Rietsch} endows $Q(W,W_P)$ with the structure of a poset.

%
%
%
%
%
%
%
%
%

\section{Frobenius splitting of projected Richardson varieties}
\label{sec:frobsplit}

In this section we show that the partial flag varieties $G/P$ possess Frobenius
splittings which compatibly split all the projected Richardson varieties
there. We will later show that the projected Richardson varieties are the only compatibly split subvarieties for this splitting, generalizing a result of
\cite{Hague} in the $P=B$ case.

This result, and a related result
from \cite{BrionLakshmibai}, will allow us to prove that the map to a projected Richardson variety from its Richardson model is ``cohomologically trivial''.  Using the result (Theorem \ref{thm:RatlRes}), established in the Appendix, that Richardson varieties have rational resolutions, we obtain that projected Richardson varieties are normal, Cohen-Macaulay, and have rational resolutions.


We will not need to define (compatible) Frobenius splittings,
as everything we will need about them is contained in
the following lemma, all parts quoted from \cite{BK}.

\newcommand\LL{{\mathcal L}} 

\begin{lem}\label{lem:splittings} \
  \begin{enumerate}
  \item If $X$ is Frobenius split, it is reduced.
  \item If $X_1,X_2$ are compatibly split subvarieties then $X_1 \cup X_2$,
    $X_1 \cap X_2$, and their components are also compatibly split in $X$.
  \item If $f:X\to Y$ is a morphism such that
    the map $f^{\#}: \O_Y \to f_* \O_X$ is an isomorphism and
    $X$ is Frobenius split, then $f$ induces a natural splitting on $Y$.

    Moreover, if $X'\subseteq X$ is compatibly split,
    then the splitting on $Y$ compatibly splits $f(X')$.

    \item There is a Frobenius splitting of $G/B$, for which all Richardson varieties are compatibly split.
  \item If $X$ is Frobenius split and proper and $\LL$ is ample on $X$,
    then $H^i(X; \LL) = 0$ for $i>0$.
  \end{enumerate}
\end{lem}

\begin{proof}
  The first two are Proposition 1.2.1,
  the third is Lemma 1.1.8,
  the fourth is Theorem 2.3.1,
  and the fifth is part (1) of Theorem 1.2.8,
  all from \cite{BK}.
\end{proof}

\begin{remark}
If $f: X \to Y$ is a projective and surjective map of reduced and
irreducible varieties with connected fibers, and $Y$ is normal, then
$f^{\#}: \O_Y \to f_* \O_X$ is an isomorphism (see
\cite[p.125]{Laz}).
\end{remark}

\begin{cor}\label{cor:pRsplit}
  There is a Frobenius splitting on $G/P$ that compatibly
  splits all the projected Richardson varieties therein.
\end{cor}

\begin{proof}
  The map $\pi : G/B \to G/P$ satisfies the hypothesis of (3).
  By parts (4) and (3) of the lemma, $G/P$ acquires a
  splitting that compatibly splits all projected Richardson varieties.
\end{proof}

\begin{cor} \label{C:reducedintersection}
  If $A$ and $B$ are unions of projected Richardson varieties, 
  then $A \cap B$ is reduced.
\end{cor}

By Proposition \ref{P:Rietsch}, this intersection is set-theoretically a union of projected Richardson varieties. 
So, combining Corollary~\ref{C:reducedintersection} and Rietsch's result, this intersection is a reduced union of projected Richardson varieties. 

\begin{theorem} \label{T:crepant}
  Let $\Pi_u^w$ be a projected Richardson variety and $X_u^w$ a Richardson model.
  Then the map $\pi : X_u^w \onto \Pi_u^w$ is \defn{cohomologically trivial},
  i.e. $\pi_* \O_{X_u^w} = \O_{\Pi_u^w}$ and $R^i \pi_* \O_{X_u^w} = 0$ for $i>0$.

  Also, for any ample line bundle $L$ on $G/P$, we have $H^0(\Pi_u^w; L) \iso H^0(X_u^w; \pi^*L)$,
  and $H^i(\Pi_u^w; L) \iso H^i(X_u^w; \pi^* L) = 0$ for $i>0$.
\end{theorem}

Shrawan Kumar has remarked that Theorem \ref{T:crepant} also follows from Lemmas 3.3.2 and 3.3.3 in \cite{BK}, which have a similar proof.

\begin{proof}
  For any ample line bundle $L$, the commuting square below on the left induces by functorality the one on the right:
  $$
  \begin{array}{ccccc}
    X_u^w & \longrightarrow & \Pi_u^w \\
    \downarrow  & & \downarrow \\
    G/B & \longrightarrow & G/P
  \end{array}
  \qquad\qquad\qquad
  \begin{array}{ccccc}
   H^i(X_u^w; \pi^* L) & \longleftarrow & H^i(\Pi_u^w; L) \\
   \uparrow  & & \uparrow \\
   H^i(G/B; \pi^* L) & \longleftarrow & H^i(G/P; L)
  \end{array}
  $$
  By Borel-Weil, we know the bottom cohomology map is an isomorphism,
  and both sides are zero for $i>0$. By \cite[Proposition 1]{BrionLakshmibai},
  the left cohomology map is a surjection. Hence the composite map
  $H^i(G/P; L) \to    H^i(X_u^w; \pi^* L)$ is a
surjection whose image is zero if $i>0$. The top cohomology map
$H^i(\Pi_u^w; L) \to  H^i(X_u^w; \pi^* L)$ is then also
  is a surjection whose image is zero for $i>0$.

  If $i=0$, this top map is injective as well,
  since $X_u^w \to \Pi_u^w$ is a surjection.
  This proves the claim that $H^0(\Pi_u^w; L) \iso H^0(X_u^w; \pi^* L)$.
  We now establish the other parts of the result.
  
  Now, let $K$ be the cokernel of $0 \to \O_{\Pi_u^w} \to \pi_* \O_{X_u^w}$.
  For $N$ sufficiently large, the sequence
  $$0 \to H^0(\O_{\Pi_u^w} \otimes L^{\otimes N}) \to H^0(\pi_* \O_{X_u^w} \otimes L^{\otimes N}) \to H^0(K \otimes L^{\otimes N}) \to 0$$
  is exact. (Here all sheaves live on $G/P$.)
  But, by \cite[Ex.~II.5.1(d)]{Har}, the middle term is $H^0(X_u^w; \pi^*L^{\otimes N})$ and we just showed that $H^0(\Pi_u^w; L^{\otimes N}) \to  H^0(X_u^w; \pi^* L^{\otimes N})$ is an isomorphism.
  So $H^0(K \otimes L^{\otimes N})=0$ for all sufficiently large $N$, and we deduce that $K$ is the zero sheaf.

 Now consider the case that $i>0$.
  Consider the Leray spectral sequence for $\pi$ and $\O_{X_u^w}  \otimes \pi^* L^{\otimes N}$.
  The $E_2$ term is $H^p((R^q \pi_*) (\O_{X_v^w} \otimes \pi^* L^{\otimes N}), G/P)$, which, by \cite[Ex. III.8.3]{Har} is $H^p((R^q \pi_*) (\O_{X_u^w}) \otimes L^{\otimes N}, G/P)$.
  We take $N$ sufficiently large that this vanishes except when $p=0$.
  So we deduce that, for $N$ sufficiently large, $H^i(\O_{X_u^w}  \otimes \pi^*L^{\otimes N}, G/B) \cong H^0((R^i \pi_*)(\O_{X_u^w}) \otimes L^{\otimes N}, G/P)$.
As we observed in the first paragraph, the left hand side is zero.
So $H^0((R^i \pi_*)(\O_{X_u^w}) \otimes L^{\otimes N}, G/P)$ vanishes
for $N$ sufficiently large. But this means that $(R^i
\pi_*)(\O_{X_u^w})$ is zero, as desired.

Finally, to see that $H^i(\Pi_u^w; L) = 0$, use part (5) of Lemma \ref{lem:splittings}.
\end{proof}

Theorem~\ref{T:crepant} has many consequences for the geometry of projected Richardson varieties, as we now describe.

\begin{cor}
Let $\Pi$ be a projected Richardson variety and $X_u^w$ a Richardson model for
$\Pi$. Then the fibers of $\pi : X_u^w \to \Pi_u^w$ are connected.
\end{cor}

\begin{proof} \label{c:ConnectedFibers}
 See \cite[Corollary~III.11.4]{Har}.
\end{proof}


\begin{cor} \label{C:Normal}
Projected Richardson varieties are normal.
\end{cor}

\begin{proof}
Let $\Pi$ be a projected Richardson variety and $X$ its Richardson model. By \cite{Bripos}, $X$ is normal. We
establish, more generally, that if $X$ is a normal variety and $\pi
: X \to \Pi$ a morphism such that $\O_{\Pi} \to \pi_* \O_X$ is an
isomorphism, then $\Pi$ is normal.

Normality is a local condition, so we may assume that $\Pi = \Spec A$
for $A$ some integral domain. Let $K$ be the fraction field of $A$
and let $x \in K$ be integral over $A$, obeying the equation $x^n =
\sum_{i=0}^{n-1} a_i x^i$ for some $a_0$, $a_1$, \dots, $a_{n-1} \in
A$. Then we also have the relation $\pi^*(x)^n = \sum_{i=0}^{n-1}
\pi^*(a_i) \pi^*(x)^i$ in $\O_X(X)$. So, since $X$ is normal,
$\pi^*(x)$ is in $\O_{\Pi}(\Pi)=\left( \pi_* \O_{X} \right)(\Pi)$. But, by
hypothesis, $\left( \pi_* \O_{X} \right)(\Pi)=\O_{\Pi}(\Pi)$, so $x$ is in
$\O_{\Pi}(\Pi)$. Thus, we see that $\Pi$ is normal.
\end{proof}

In characteristic zero, one defines a variety $V$ to have \defn{rational singularities} if there is
a smooth variety $W$ and a proper birational map $p: W \to V$ such
that $p_* \O_W \to \O_V$ is an isomorphism and $R^i p_* \O_W=0$ for $i>0$.  In positive characteristic, the correct notion is that of a \defn{rational resolution}: a smooth variety $W$ and a proper birational map $p: W \to V$ such
that $p_* \O_W \simeq \O_V$ and $p_* \omega_W \simeq \omega_V$ and $R^i p_* \O_W=0 = R^ip_*\omega_V$ for $i>0$.

\begin{thm} \label{T:RatSing}
  Projected Richardson varieties have rational resolutions.
\end{thm}

\begin{proof}
By Theorem \ref{thm:RatlRes} in the Appendix, Richardson varieties have rational resolutions.
Let $\Pi$ be a projected Richardson variety, $X$ its Richardson model, and
$\psi : Z \to X$ the rational resolution of $X$ of Theorem \ref{thm:RatlRes}. Consider the map
$\psi \circ \pi: Z \to \Pi$. Since $\psi$ and $\pi$ are
proper and birational, so is $\psi \circ \pi$. By functoriality of
pushforward, $(\psi \circ \pi)_* \O_Z \to \O_X$ is an
isomorphism. From the Grothendieck spectral sequence \cite[\S 5.8]{Wei},
and the knowledge that $R^i \psi_* \O_Z$ and $R^i \pi_* \O_X$
vanish, we know that $R^i (\psi \circ \pi)_* \O_Z$ vanishes for $i > 0$.  
In the notation of Appendix \ref{sec:appendix}, the map $(\psi \circ \pi)|_{\Zo}: \Zo \to \Pio$ is an isomorphism, so the boundary divisor $\Delta \subset Z$ contains the exceptional locus of $\psi \circ \pi$.
So~\cite[Theorem~1.3.14]{BK} applies, and we deduce that $R^i (\psi\circ \pi)_* \omega_Z = 0$ for $i>0$.  Finally, \cite[Lemma 3.4.2]{BK}, states that the isomorphism $(\psi \circ \pi)_* \omega_Z \simeq \omega_X$ follows from the other statements.
\end{proof}

\begin{cor} \label{c:CohenMacaulay}
  Projected Richardson varieties are Cohen-Macaulay.
\end{cor}

\begin{proof}
By \cite[Proposition 4]{Ram}, projective varieties with rational resolutions are automatically Cohen-Macaulay.
\end{proof}

We end with a caution. Many of these results show that projected Richardson varieties are like Richardson varieties. 
Here is a way in which they are not:

\begin{Example}
  Let $L$ be an ample line bundle on $G/P$, and assume $G/P$ is
  minuscule.  Let $\Pi \subseteq G/P$ be a projected Richardson
  variety.  Let $S$ be the projective coordinate ring,
  $\bigoplus_{n=0}^{\infty} H^0(G/P, L^{\otimes n})$.  Let $I \subset S$ 
  be the homogenous ideal of $\Pi$. If $\Pi$ is actually a
  Richardson variety, then by \cite{Ramanathan} $I$ is generated in degree $1$,
  but in general it may not be.

Specifically, let $G=\Sp_4$, the group of symmetries of $k^4$ preserving the symplectic form $e_1 \wedge e_3+ e_2 \wedge e_4$. 
Then $G/B$ is the space of pairs $(L,M)$ where $M$ is an isotropic $2$-plane in $k^4$ and $M$ is a line in $L$. 
Let $G/P$ be the partial flag variety where we forget $M$, so $G/P \cong \PP^3$.
Consider the Richardson variety $X$ in $G/B$ where we require that $M$ meets the isotropic planes $L_1:=\mathrm{Span}(e_1, e_2)$ and $L_2:= \mathrm{Span}(e_3, e_4)$.
The projection $\Pi$ of this Richardson is the set of points which lie on an isotropic line joining $\PP(L_1)$ to $\PP(L_2)$. 
For $L=(p_1: p_2: p_3: p_4)$ of $\PP^3$, the unique line through $L$, $L_1$ and $L_2$ is the one which meets $L_1$ at $(p_1:p_2:0:0)$ and  $L_2$ at $(0:0:p_3:p_4)$.
This line is isotropic if and only if $p_1 p_3 + p_2 p_4=0$.
So $\Pi$ is a quadric, cut out by the equation $p_1 p_3+ p_2 p_4=0$.
\end{Example}

%
%

\section{Projected Richardson varieties are the only split
  subvarieties}

%

Let $\phi$ be the Frobenius splitting on $G/B$ which splits the
Richardson varieties \cite[Theorem 2.3.1]{BK}, and let $\pi_* \phi$ be the
induced splitting on $G/P$. The purpose of this section is to establish the following theorem:

\begin{theorem} \label{T:OnlySplit}
The compatibly split subvarieties of $G/P$, with respect to the splitting $\pi_* \phi$, are precisely the projected Richardson varieties.
\end{theorem}

There is an as yet imprecise analogy between Frobenius splitting and semi-classical deformation. 
Our Theorem~\ref{T:OnlySplit} is analogous to results of Goodearl and Yakimov, see~\cite{GY} and~\cite{Yakimov}.

The theorem below will be our way of characterizing the compatibly split
subvarieties in a Frobenius split variety, under very special hypotheses
that we verify in the case of $G/P$. Call a divisor $D$ in a
normal variety $X$ \defn{anticanonical} if $D\cap X_{reg}$ is 
anticanonical in the regular locus $X_{reg}$.

\begin{lem}\label{lem:onlyanticanonical}
  Let $X$ be complete and normal, and let $D$ be a divisor containing
  an anticanonical divisor (meaning, $D\cap X_{reg}$ is anticanonical
  plus effective).  If there is a splitting $\phi$ of $X$ that
  compatibly splits $D$, then $D$ is anticanonical, and $\phi$ splits
  no other proper subvarieties of $X_{reg} \setminus D$.
\end{lem}

\begin{proof}
  Since $X$ is normal, we may apply \cite[Proposition 1.3.11 and
  Remark 1.3.12]{BK}.  Let $E$ be the divisor in $X$ corresponding
  to the section of $\omega_{X}^{-(p-1)}$ over $X_{reg}$ giving the splitting of $X$.  
  By \cite[Proposition 2.1]{KumarMehta}, 
  if $D$ is a split divisor in $X$, then $E$ vanishes to order $p-1$ on $D$,
  i.e. $E - (p-1)D$ is effective.

  If $D = A + D'$ where $A$ is anticanonical and $D'$ is effective,
  then then the effective divisor $E - (p-1)D$ is equivalent 
  to the anti-effective divisor  $-(p-1) D'$.   
  Since $X$ is complete, $E - (p-1)D$ and $D'$ must each be empty.
  
  Again by \cite[Proposition 2.1]{KumarMehta}, any proper compatibly
  split subvariety of $X_{reg}$ must lie inside $E$, whose support is $D$.
\end{proof}

\newcommand\calY{{\mathcal Y}}
\newcommand\union{\cup}

That lemma lets us simplify slightly the argument from \cite{Hague},
as follows.

\begin{thm}\label{thm:allsplit}
  Let $X$ be complete, normal, and Frobenius split. Let $\calY$ be a
  finite collection of compatibly split (irreducible) subvarieties of $X$ defining a
  stratification, i.e.  $X\in \calY$ and the intersection $Y_1\cap Y_2$ 
  of any two closed strata must be a union of others. Assume that
  \begin{enumerate}
  \item each closed stratum $Y\in \calY$ is normal,
  \item each open stratum $Y \setminus \union_{Z\in \calY, Z\subsetneq Y} Z$
    is regular, and
  \item $\partial X := \union_{Y\in \calY, \codim_X Y = 1} Y$ is an anticanonical
    divisor in $X$.
  \end{enumerate}
  Then $\calY$ contains all the compatibly split subvarieties in $X$,
  and for each $Y\in\calY$, $\union_{Y'\in \calY, Y' \subsetneq Y} Y'$ 
  is an anticanonical divisor.
\end{thm}

It was recently proven in \cite{Schwede,KumarMehta}
that on a variety with a fixed Frobenius splitting, there are only
finitely many compatibly split subvarieties, so the finiteness condition
on $\calY$ is automatic.

\begin{proof}
  Let $(X,Z)$ be a minimal counterexample, in that $Z\notin \calY$ is
  a compatibly split subvariety in $X$, with $\codim_X Z$
  minimized. Since $X\in\calY$, we know $Z \subsetneq X$.

  So we claim $Z \subseteq \union_{Y\in \calY, Y\subsetneq X} Y$.
  This uses Lemma \ref{lem:onlyanticanonical}, assumption (3), 
  and the fact that the open stratum 
  $X\setminus \union_{Y\in \calY, Y\subsetneq X} Y$ is regular (assumption (2)).

  Since $Z$ is irreducible, and $\union_{Y\in \calY, Y\subsetneq X} Y$
  is a finite union by assumption on $\calY$, we have that $Z$ is
  contained in some divisor $X'\in \calY$.

  If we can show that $X'$ satisfies the assumptions on $X$, then the
  pair $(X',Z)$ will be a smaller counterexample, contradicting minimality.
  Assumptions (1) and (2) are clear. To find a split anticanonical divisor
  $\partial X'$ inside $X'$, note that $\overline{\partial X \setminus X'}$ 
  is a union of strata, 
  and take $\partial X' := X' \cap \overline{\partial X \setminus X'}$.
  By the adjunction formula, $\partial X'$ is anticanonical in $X'$.
  It is also contained in $\union_{Y\in \calY, Y\subsetneq X'} Y$, 
  so by Lemma \ref{lem:onlyanticanonical}'s conclusion on $D$,
  the two are equal.  
\end{proof}

For a particularly simple example of this theorem, let $X$ be the
projective toric variety associated to a polytope $P$, $\phi$ its
standard splitting, and $\calY$ the
set of toric subvarieties (associated to the faces of $P$). 
Then $\partial X$ is the toric subscheme associated to the boundary 
$\partial P$ in its usual sense, and the divisor $X'$ will be associated 
to some facet $F$ of $P$. 
If we divide this spherical boundary $\partial P$ into 
the discs $F$ and $\overline{\partial P \setminus F}$, 
their intersection is the boundary of $F$, in parallel
with the adjunction formula calculation in the proof.

Conditions (1) and (2) definitely do not hold in general. For example,
there is a splitting on $\PP^2$ that compatibly splits a
nodal elliptic curve, and a splitting on $\PP^3$ that splits a normal quartic
surface with an isolated singular (and split) point.
So we are lucky to be able to apply this theorem in
the case of the standard splitting on $G/P$.

\begin{lemma} \label{L:Hyp1} 
  Let $\Pi$ be any projected Richardson variety (though we will only
  need the case $\Pi = G/P$). Let $\Pi_1$, $\Pi_2$, \dots, $\Pi_r$ be
  those projected Richardson varieties which are hypersurfaces in $\Pi$.  
  Then $\sum [ \Pi_i ]$ is an anticanonical divisor in $\Pi_u^w$.
\end{lemma}

Note that the proof of this is quite easy in the case $\Pi = G/B$,
where (under the identification of $Pic(G/B)$ with $T$'s weight lattice)
the anticanonical class is the sum of the positive roots.
The relevant (non-projected) Richardson varieties are the
Schubert divisors and opposite Schubert divisors, adding up
to twice the sum of the fundamental weights, which matches
the sum of the positive roots.

\begin{proof}
We imitate the proof of \cite[Proposition 2.2.7(ii)]{Bri}.  Let $X_u^w$ be the Richardson model for $\Pi_u^w$ for which $w \in W^P$.  We will abbreviate these to $X$ and $\Pi$ when possible.
By Theorem~\ref{T:RatSing} and \cite[Proposition 2.2.5]{Bri}, $\pi_* \omega_{X} = \omega_{\Pi}$. 
By \cite[Theorem 4.2.1(i)]{Bri}, the divisor $D_X = \sum_{u \lessdot u' \leq w} [X_{u'}^w] + \sum_{u \leq w' \lessdot w} [X_{u}^{w'}]$ is  anticanonical in $X$, so that $\omega_X \simeq \O_X(-D)$.  

Let $D_\Pi$ be the divisor $\sum_{u \lessdot u' \leq_P w} [\Pi_{u'}^w] + \sum_{u \leq_P w' \lessdot w} [\Pi_{u}^{w'}]$.  By Proposition \ref{P:Rietsch}, these are exactly the projected Richardson hypersurfaces in $\Pi$, and furthermore, one has $\pi^{-1}(D_\Pi) = D_X$.  Thus $\omega_\Pi = \pi_*(\omega_X) = \pi_*(\O_X(-D_X)) = \O_\Pi(-D_\Pi)$.

\end{proof}

\begin{proof}[Proof of Theorem \ref{T:OnlySplit}]
  Let $\calY$ be the set of projected Richardson varieties.
  We know they are compatibly split by Corollary \ref{cor:pRsplit},
  and normal by Corollary \ref{C:Normal},
  with smooth interior by Corollary \ref{C:ProjectionDimension}.
  We know they form a stratification by \cite[Proposition 7.2]{Rie}.
  For assumption (3), we appeal to the $\Pi = G/P$ case of Lemma \ref{L:Hyp1}.
  We have verified the assumptions of Theorem \ref{thm:allsplit}.
\end{proof}

\section{Shelling $\pi(\Delta([u,w]))$}\label{sec:shelling}
\subsection{Shellings}
Given a simplicial complex with maximal faces $\sigma_1$, $\sigma_2$, \dots, $\sigma_N$ all of dimension $d-1$, the ordering  $\sigma_1$, $\sigma_2$, \dots, $\sigma_N$ is called a \defn{shelling order} if, for each $i$, $\sigma_i \cap \bigcup_{j < i } \sigma_j$ is pure of dimension $d-2$.  A simplicial complex is called \defn{shellable} if its maximal faces can be put into a shelling order. For a (finite) poset $Q$, we let $\Delta(Q)$ denote the simplicial complex consisting of the chains in $Q$, called the \defn{order complex} of $Q$.  

A labeling of the Hasse diagram of a poset $Q$ by
some totally ordered set $\Lambda$ is called an \defn{$EL$-labeling}
if for any $x \leq y \in Q$:
\begin{enumerate}
\item
  there is a unique strictly-label-increasing saturated chain $C$
  from $x$ to $y$,
\item
  the sequence of labels in $C$ is $\Lambda$-lexicographically minimal amongst
  the labels of saturated chains from $x$ to $y$.
\end{enumerate}
If $Q$ has an $EL$-labeling then we say that $Q$ is 
\defn{$EL$-shellable}.  Dyer \cite[Proposition 4.3]{Dye} showed that every Bruhat order (and also its dual) is
$EL$-shellable.  (See also~\cite{BW}.)  This implies (\cite{Bjo}) that the order complex 
of the Bruhat order is shellable. 

%

In the following we will use $W$ to denote both a Coxeter group, and also its the Bruhat order poset.  We fix a parabolic subgroup $W_P \subset W$, and let $\pi:W \to W^P$ denote the natural projection.  Our aim in this section is to establish the following result:

\begin{theorem} \label{theorem shellable}
For any interval $[u,w]$ in $W$, the simplicial complex $\pi(\Delta([u,w]))$ is shellable. 
\end{theorem}

Note that $\pi(\Delta([u,w]))$ is a simplicial complex on $W^P$.  Our
intended application is to the case where $G/P$ is minuscule, as
defined in \S \ref{sec:intro}.  Our proof, however, is valid for all
$(W, W_P)$, even if $W$ is infinite or non-crystallographic.

\subsection{Some preliminaries}
We will need several results of Dyer on reflection orders. Let $V$ denote the reflection representation of $W$ and $T$ denote the set of reflections in $W$. To each reflection $t$ is associated a positive root $\beta_t$ in $V$.\footnote{We will only care about roots up to positive rescaling, so we don't need to discuss any subtleties about the choice of Cartan matrix, or the definition of non-crystallographic root systems. If the reader wants a definite choice, normalize all roots to have length $\sqrt{2}$.} 
Let $H$ be a two-dimensional subspace of $V$ containing at least two positive roots. Then there is a unique pair of reflections, $p$ and $q$, such that every positive root in $H$ is in the positive span of $\beta_p$ and $\beta_q$. The positive roots in $V$ will correspond to the reflections $p$, $pqp$, $pqpqp$, $pqpqpqp$, \dots, $qpqpqpq$, $qpqpq$, $qpq$, $q$. This may either be an infinite sequence or a finite one. (In the latter case, $(pq)^m=1$ and the sequence has $m$ terms.) A \defn{reflection order} is a total ordering $\prec$ on $T$ such that, for every $p$ and $q$ as above, we either have $p \prec pqp \prec pqpqp \prec \ldots \prec qpq \prec q$, or vice versa.

We need the following results of Dyer:

\begin{prop}[{\cite[Proposition 2.3]{Dye}}] \label{prop P last}
There is a reflection ordering such that any reflection in $W_P$ comes after any reflection not in $W_P$.  There is a reflection ordering such that any reflection in $W_P$ comes before any reflection not in $W_P$.
\end{prop}

\begin{prop}[{\cite[Lemma 4.1]{Dye}}] \label{prop diamond}
Let $a \lessdot ap, ar \lessdot apq = ars$ be a length-two interval in $W$ and let $\prec$ be any reflection order. Then either $p, s \prec q, r$ or $q, r \prec p, s$.
\end{prop}

Let $c_0 \lessdot c_1 \lessdot \ldots \lessdot c_{\ell}$ be any saturated chain in $W$. Then we will say that the reflection sequence of this chain is $(c_0^{-1} c_1, c_1^{-1} c_2, \ldots, c_{\ell-1}^{-1} c_{\ell})$. 

We write $\prec_{lex}$ for the lexicographic order induced by $\prec$.

\begin{prop}[{\cite[Proposition 4.3]{Dye}}] \label{prop lex inc}
Let $[a,b]$ be any interval in $W$ and let $\prec$ be any reflection order. Then there is precisely one chain $a= c_0 \lessdot c_1 \lessdot \ldots \lessdot c_{\ell} =b$ whose associated reflection sequence $(\lambda_1, \lambda_2, \ldots, \lambda_{\ell})$ satisfies $\lambda_1 \prec \lambda_2 \prec \ldots \prec \lambda_{\ell}$. Moreover, if we have any other chain from $a$ to $b$, with associated reflection sequence $(\mu_1, \mu_2, \ldots, \mu_{\ell})$, then $(\lambda_1, \ldots, \lambda_{\ell})  \prec_{lex} (\mu_1, \mu_2, \ldots, \mu_{\ell})$.
\end{prop}

We call the unique chain $c_{\bullet}$ the \defn{increasing chain}, or the \defn{lexicographically minimal chain}, depending on which of its properties we want to emphasize.  From now on, $\prec$ will denote a fixed reflection order which puts reflections in $W_P$ after reflections not in $W_P$, as in Proposition~\ref{prop P last}.  Also let $\prec'$ denote a reflection order which puts reflections in $W_P$ before reflections not in $W_P$.

We also need the following easy lemmas about parabolic cosets:

\begin{lem} \label{L:par interval}
If $xW_P = yW_P$, then the whole interval $[x,y]$ lies in the coset $xW_P$.
\end{lem}

\begin{proof}
If $x \leq z \leq y$ then $x^P \leq z^P \leq y^P$. Since $x^P=y^P$, we have $z^P=x^P$.
\end{proof}

Recall that the Demazure product $\circ$ was defined in \S \ref{sec:proj}.

\begin{lem} \label{L:upper bound}
If $xW_P=yW_P$, then there exists $z \in xW_P$ with $x, y \leq z$.
\end{lem}
\begin{proof}
Let $m$ be the minimal element of $xP$ and write $x=mx'$ and $y = my'$. Then $z=m(x' \circ y')$ has the required property.
\end{proof}

\begin{proof}[Proof of Proposition \ref{P:uniquelift}]
Let $z, z' \in C_{\geq x}$ be two minimal elements.  By Lemma \ref{L:upper bound} there exists $y \in C_{\geq x}$ which is an upper bound for $z$ and $z'$.  We shall prove that $z = z'$ using induction on $\ell(y) - \max(\ell(z),\ell(z'))$.  If $\ell(y) = \ell(z)$ (resp. $\ell(y) = \ell(z')$), then it is clear that $z = z'$, so the base case is trivial.  We now suppose that $\ell(y) > \ell(z),\ell(z')$.

Let $c = (x \lessdot \cdots \lessdot z \lessdot \cdots \lessdot w \lessdot y)$ and $c' = (x  \lessdot \cdots \lessdot z' \lessdot \cdots \lessdot w' \lessdot y)$ be two saturated chains from $x$ to $y$ going through $z$ and $z'$.  We may, and will, assume that $w$ and $w'$ are both in $C$.  We claim that $c$ can be changed to $c'$ via a sequence of saturated chains such that at each step only one element of the chain changes.  Furthermore, we will choose such a sequence of chains
\begin{align*}
c = c_0& = (x \lessdot \cdots \lessdot z \lessdot \cdots \lessdot u_0 \lessdot w = w_0 \lessdot y) \\ 
c_1&= (x \lessdot\cdots \cdots \cdots \cdots \lessdot u_1 \lessdot w = w_1 \lessdot y) \\
c_2&= (x \lessdot \cdots \cdots \cdots \cdots\lessdot u_2 \lessdot w = w_2 \lessdot y)
\\
\cdots \\
c'=c_N&=  (x  \lessdot \cdots \lessdot z' \lessdot \cdots \lessdot u_N \lessdot w'=w_N \lessdot y)
\end{align*}
where $w_i \in C$.  To see this is possible we use Dyer's theorem \cite[Proposition 4.3]{Dye} that $\prec'$ gives a shelling order on $[x,y]$ where now we order maximal chains using $\prec'$ lexicographically from the top of the chain (see \cite[(4.4)]{Dye}).  Dyer's result implies that we may find saturated chains $c_i$ from $x$ to $y$ so that
$$
c \succ'_{\lex} c_1 \succ'_{\lex} c_2 \cdots \succ'_{\lex} c_r \prec'_{\lex} c_{r+1} \prec'_{\lex} \cdots \prec'_{lex} c'
$$
where $\prec'_{\lex}$ is the lexicographic order on chains induced by $\prec'$ starting from the top of the chain, and $c_i$ and $c_{i+1}$ differ by one element.  It follows from the definition of $\prec'$ that these $c_i$ have the stated property.

Now let us suppose that $w_i \neq w_{i+1}$ for some $0 \leq i \leq N-1$ so that the end of $c_i$ and $c_{i+1}$ look like $u \lessdot w_i \lessdot y$ and $u \lessdot w_{i+1} \lessdot y$ respectively.  Applying Proposition \ref{prop diamond} with the reflection order $\prec$ one deduces that $u \in C$.  It follows that we may assume that $\ell(y) \geq \max(\ell(z),\ell(z')) + 2$, and so we shall now in addition assume that all the $u_i$ lie in $C$.  Let $z_i \in C_{\geq x}$ be a minimum element below $u_i$, such that $z = z_0$ and $z' = z_N$.  For each $i$, either (1) $u_i = u_{i+1}$ which implies that $u= u_i = u_{i+1}$ is an upper bound for $z_i$ and $z_{i+1}$, or (2) $v_i = v_{i+1}$, which implies that $v = v_i = v_{i+1}$ is an upper bound for $z_i$ and $z_{i+1}$.  In either case, by induction we deduce that $z_i = z_{i+1}$, and thus $z = z_0 = z_1 = \cdots = z_N = z'$. 

For the final statement, we consider the increasing saturated chain from $x$ to the minimum $z \in C_{\geq x}$ under the order $\prec$.  By the minimality of $z$, this chain does not use any reflections in $W_P$, and hence is a saturated $P$-Bruhat order chain.  Thus $x \leq_P z$.
\end{proof}

\begin{lemma}\label{L:identical}
Suppose $[u,v]_P \sim [x,y]_P$.  Then the complexes $\pi(\Delta([u,v]_P))$ and $\pi(\Delta([x,y]_P))$ are identical.  The complexes $\pi(\Delta([u,v]))$ and $\pi(\Delta([x,y]))$ are also identical.
\end{lemma}
\begin{proof}
By Lemma \ref{L:lengthadd}, it suffices to prove the case that $x = us_i$ and $y=vs_i$ for $s_i \in W_P$ where $\ell(x) = \ell(u) + 1$ and $\ell(y) = \ell(v) + 1$.  The proof of Lemma \ref{L:lengthadd} shows that there is a bijection between maximal chains of $[u,v]_P$ and $[x,y]_P$ which preserves the images under $\pi$.  This establishes the first statement.

For the second statement, let $u =w_0 \lessdot w_1 \lessdot w_2 \lessdot \cdots \lessdot w_{r-1} \lessdot w_r = v$ be a maximal chain in $[u,v]$.  We now define a chain $x = w'_0 \lessdot w'_1 \lessdot \cdots \lessdot w'_{r-1} \lessdot w'_r = y$ as follows.  Let $w'_0 = us_i$.  Suppose $w'_i$ is defined.  Then we let $w'_{i+1}$ be either (1) $w_{i+1}s_i$, if $w_{i+1}s_i \gtrdot w_{i+1}$ or (2) $w_{i+2}$, if $w_{i+1}s_i \lessdot w_{i+1}$. In case (1), we continue the
recursive construction; in case (2), we set $w'_j = w_{j+1}$ for $r > j > i + 1$, and let $w'_r = ws_i$.  Since $w_is_i \gtrdot⋗w_i$, by \cite[Proposition 5.9]{Hum}, case (2) only occurs if $w_{i+1} = w_is_i$.  In particular
it follows that $w'_{i+1} = w_{i+2} \gtrdot w_{i+1} = w'_i$ and that $w_i$ and $w_{i+1}$ have the same image under $\pi$.  It follows that the two chains have the same image under $\pi$.  This shows that $\pi(\Delta([u,v])) \subset \pi(\Delta([x,y]))$.  The reverse inclusion is established in a similar manner.
\end{proof}

\begin{lemma}
Suppose $u \leq_P v$.  Then the complexes $\pi(\Delta([u,v]))$ and $\pi(\Delta([u,v]_P))$ are identical.  Furthermore, the complex $\pi(\Delta([u,v]))$ is pure dimensional, of dimension $\ell(v) - \ell(u)$. 
\end{lemma}
\begin{proof}
By Proposition \ref{P:TopBottom} and Lemma \ref{L:identical}, it suffices to establish the claim in the case that $v \in W^P$.  Let $(a_1,a_2,\ldots,a_l)$ be a maximal face of $\pi(\Delta([u,v]))$, and let $w_1 < w_2 < \cdots < w_l$ be a chain in $[u,v]$ mapping to $(a_1,a_2,\ldots,a_l)$.  We assume that $\pi(w_i) = a_i$, and that $u=w_1$.  Let us define $w'_i$ recursively as follows.  Set $w'_1 = w_1 = u$, and let $w'_i$ be the minimal element of the set $(w_{i} W_P)_{\geq w'_{i-1}}$.  Since $\pi(w'_{i-1}) = \pi(w_{i-1}) < \pi(w_i)$ a unique such element exists by Proposition \ref{P:uniquelift}, and furthermore one has $w'_{i-1} \leq_P w'_i$.  But it is clear that $v\geq w_i \geq w'_i$, and so by Proposition \ref{P:TopBottom}, $w'_1 \leq_P w'_2 \leq_P \cdots \leq_P w'_r$ is a chain in $[u,v]_P$.  This establishes the first statement.  The second statement follows from the fact that maximal faces of $\pi(\Delta([u,v]_P))$ are precisely the images of the maximal chains of $[u,v]_P$ which all have length $\ell(v)-\ell(u)$.
\end{proof}

Define the \defn{downwards Demazure product} $\circ'$ of $W$
by
$$
w \circ' s_i  = \begin{cases} ws_i & \mbox{if $ws_i < w$,} \\
w &\mbox{otherwise}
\end{cases}
$$
and extend it to $w\circ' v$ using any reduced word for $v$.

\begin{lemma}\label{L:downDemazure}
Suppose $u \leq v$.  Then $\pi(\Delta([u,v]))$ and $\pi(\Delta([u\circ' v_P, v^P]))$ are identical.
\end{lemma}
\begin{proof}
Let $s_i \in W_P$.  It suffices to establish two cases: (1) $\pi(\Delta[u,v])) = \pi(\Delta([us_i,vs_i])$ where $us_i < u$ and $vs_i < v$, and (2)
$\pi(\Delta([u,v])) = \pi(\Delta([u,vs_i])$ where $us_i > u$ and $vs_i < v$.  Case (1) follows from the argument in the proof of Lemma \ref{L:identical}.  For Case (2), the inclusion $\pi(\Delta([u,v])) \subset \pi(\Delta([u,vs_i])$ also follows from the argument in the proof of Lemma \ref{L:identical}, and the other inclusion is obvious.
\end{proof}

Thus to study a general $\pi(\Delta([u,v]))$ we may restrict ourselves to pairs $u\leq_P v$ such that $v \in W^P$.

\subsection{Proof of Theorem \ref{theorem shellable}.}
The maximal faces of $\pi(\Delta([u,v]))$ where $u \leq_P v$ are precisely the images of the saturated $P$-Bruhat chains. We next show that these maximal facets are, in fact, in bijection with the $P$-Bruhat chains. 
As such, we can view the complex $\pi(\Delta([u,v]))$ as a 
the order complex of the $P$-Bruhat interval $[u,v]$ 
\emph{with some additional identifications of lower-dimensional faces.} 
Without these identifications, the order complex is generally not shellable
\cite[\S B.7]{BS}.

\begin{prop} \label{prop lifting}
Suppose that $u \leq_P v$.  The map from $P$-Bruhat chains to facets of $\pi(\Delta([u,v]))$ is bijective.
\end{prop}

\begin{proof} 
The map is clearly surjective, so we must prove injectivity. Suppose for the sake of contradiction, that $w_{\bullet}$ and $w'_{\bullet}$ were two $P$-Bruhat chains with the same projection to $W^P$. Let $i$ be minimal such that $w_i \neq w'_i$. Then $w_i W_P = w'_i W_P$; by Lemma~\ref{L:upper bound}, we can find $z \in w_i W_P$ with $w_i, w'_i \leq z$. Form the  saturated chains $w_{i-1} \lessdot w_i \lessdot x_1 \cdots \lessdot z$ and $w_{i-1} \lessdot w'_i \lessdot x'_1 \cdots \lessdot z$ so that $ w_i \lessdot x_1 \cdots \lessdot z$ is the increasing chain from $w_i$ to $z$ and  $w'_i \lessdot x'_1 \cdots \lessdot z$ is the increasing chain from $w'_i$ to $z$. But then $w_{i-1}^{-1} w_i$ is not in $W_P$ and $w_i^{-1} x_1$ is in $W_P$ (the latter by Lemma~\ref{L:par interval}) so  $w_{i-1}^{-1} w_i \prec w_i^{-1} x_1$. We see that $w_{i-1} \lessdot w_i \lessdot x_1 \cdots \lessdot z$ is an increasing chain from $w_{i-1}$ to $z$. But $w_{i-1} \lessdot w'_i \lessdot x'_1 \cdots \lessdot z$ is also an increasing chain from $w_{i-1}$ to $z$, so we have a contradiction.
\end{proof}

For simplicity, we assume (using Proposition \ref{P:TopBottom} and Lemmata \ref{L:identical} and \ref{L:downDemazure}) that $u \leq_P v$, and that $v \in W^P$.  We claim that the images of the $P$-Bruhat chains of $[u,v]_P$, ordered by $\prec_{\lex}$, give a shelling of $\pi(\Delta([u,v]))=\pi(\Delta([u,v]_P))$. 

By standard arguments, we need only check the following. Let $x_{\bullet}$ and $y_{\bullet}$ be two $P$-Bruhat chains in $[u,v]$ with $x_{\bullet} \prec_{\lex} y_{\bullet}$. Then there is a $P$-Bruhat chain $y'_{\bullet}$ with $\pi(y'_{\bullet}) \cap \pi(y_{\bullet} ) \supseteq \pi(x_{\bullet}) \cap \pi(y_{\bullet})$ and $\# \left( \pi(y'_{\bullet} ) \cap \pi(y_{\bullet} ) \right) = \ell(v) - \ell(u) -1$.

Let $i$ be minimal such that $x_i \neq y_i$. Then let $j$ be the first index larger than $i$ for which $\pi(y_j) \in \pi(x_{\bullet})$. 
\begin{lemma} 
The chain $y_{i-1} \lessdot y_i \lessdot \ldots \lessdot y_{j}$ is \emph{not} increasing.
\end{lemma}

\begin{proof}
Suppose otherwise, for the sake of contradiction. Since $y_j \geq y_{i-1} = x_{i-1}$ and $\pi(y_j) \in \pi(x_{\bullet})$, we have $\pi(y_j) \geq \pi(x_i)$.  By Lemma~\ref{L:upper bound}, we can find $z \in y_j W_P$ such that $z \geq x_i, y_j$. Now, consider the chain $y_{i-1} \lessdot y_i \lessdot \cdots \lessdot y_j \lessdot \cdots \lessdot z$ formed by concatenating $y_{i-1} \lessdot y_i \lessdot \cdots \lessdot y_j$ with the increasing chain from $y_j$ to $z$. Since reflections of $W_P$ are final in $\prec$, this chain is increasing. But it is lexicographically greater than any chain of the form $y_{i-1} \lessdot x_i \lessdot \cdots \lessdot z$, contradicting Proposition~\ref{prop lex inc}. (We use that $x_{\bullet} \prec_{\lex} y_{\bullet}$ to see that $y_{i-1}^{-1} x_i \prec y_{i-1}^{-1} y_i$.)
\end{proof}

So, there is some $r$, $i \leq r < j$, such that $y_{r-1}^{-1} y_r \succ y_r^{-1} y_{r+1}$. Write $y^0_{\bullet}$ for $y_{\bullet}$. Let $y^1_{\bullet}$ be the unique chain which agrees with $y^0_{\bullet}$ in every position but position $r$. By Proposition~\ref{prop diamond}, we have $y^1_{\bullet} \prec_{\lex} y^0_{\bullet}$ and $(y^1_{r-1})^{-1} y^1_r \prec (y^1_r)^{-1} y^1_{r+1}$. 

Now, $(y^1_{r-1})^{-1} y^1_r$ and $(y^1_r)^{-1} y^1_{r+1}$ cannot both be in $W_P$, as $\pi(y^1_{r-1}) \neq \pi(y^1_{r+1}) $. If neither of these reflections is in $W_P$, we take $y'_{\bullet} = y^1_{\bullet}$. Otherwise, by the inequality  $(y^1_{r-1})^{-1} y^1_r \prec (y^1_r)^{-1} y^1_{r+1}$, we must have $(y^1_r)^{-1} y^1_{r+1} \in W_P$. In this case, $(y^1_r)^{-1} y^1_{r+1} \succ (y^1_{r+1})^{-1} y^1_{r+2}$. Let $y^2_{\bullet}$ agree with $y^1_{\bullet}$ everywhere but in position $r+1$. Then, as before, $y^2_{\bullet} \prec_{\lex} y^1_{\bullet}$ and either $y^2_{\bullet}$ is a $P$-Bruhat chain or  $(y^2_{r+1})^{-1} y^2_{r+2} \in W_P$ and $(y^2_{r+1})^{-1} y^2_{r+2} \succ (y^2_{r+2})^{-1} y^2_{r+3}$. Continuing in this way, build $y^3$, $y^4$, etcetera. The process must halt before we reach the top of the chain, because $v\in W^P$, so we can not have $\pi(y^s_{l-1}) = \pi(y^s_l) $ with $y^{s}_{l-1} \lessdot y^s_l=v$. We take $y'_{\bullet}$ to be the $P$-Bruhat chain that results; say $y'_{\bullet} = y_{\bullet}^m$. So $y'_{\bullet} \prec_{\lex} y_{\bullet}$, as desired. It is also clear that $\# \left( \pi(y'_{\bullet} )\cap \pi(y_{\bullet} ) \right) = \ell(v) - \ell(u) -1$.

We clearly have $\pi(y'_{\bullet}) \supset \pi(y^{m-1}_{\bullet}) = \pi(y^{m-2}_{\bullet} ) = \ldots = \pi(y^1_{\bullet})$. Now, $\pi(y^1_{\bullet} )= \pi(y_{\bullet} ) \setminus \{ \pi(y_r) \}$. By construction, $i \leq r < j$, so $\pi(y_r) \not \in \pi(x_{\bullet} )$. We deduce that $\pi(y'_{\bullet}) \cap \pi(y_{\bullet} ) \supseteq \pi(x_{\bullet} )\cap \pi(y_{\bullet})$, as desired.

\subsection{Thinness}
For later use, we also establish that $\pi(\Delta([u,w]))$ is ``thin''.

\begin{prop} \label{P:Thin}
Let $u\leq_P w$.  If $F$ is a face of $\pi(\Delta([u,w]))$ of dimension   $\ell(w)-\ell(u)-1$, called a \defn{ridge}, then $F$ lies in either one or two   $(\ell(w)-\ell(u))$-dimensional faces of $\pi(\Delta([u,w]))$, and is called \defn{exterior} or \defn{interior} respectively.
  If $F$ lies in $\pi(\Delta([u',w']))$, with $\langle u',w' \rangle > \langle u,w \rangle$
  in $Q(W,W_P)$, then $F$ lies in an exterior ridge of   $\pi(\Delta([u,w]))$.
\end{prop}

\begin{proof}
Since by Proposition~\ref{prop lifting}, $\pi(\Delta([u,w]))$ is pure of dimension $\ell(w)-\ell(u)$,
we know that $F$ lies in at least one $(\ell(w)-\ell(u))$-dimensional face of $\pi(\Delta([u,w]))$.

Let the vertices of $F$ be $M_1 < M_2 < \cdots < M_l$ where the inequalities are in the partial order on $W/W_P$.
Suppose (for the sake of contradiction) that $F$ lies in three maximal faces of $\pi(\Delta([u,w]))$ with the additional vertices
$\alpha$, $\beta$ and $\gamma$.
Let $M_{a-1} < \alpha < M_a$, $M_{b-1} < \beta < M_b$ and $M_{c-1} < \gamma < M_c$
with $a \leq b \leq c$.
By Proposition~\ref{prop lifting}, each of these faces lifts to a unique chain in $[u,w]$;
let these lifts be $x_1 \lessdot x_2 \lessdot \ldots \lessdot x_{l+1}$, $y_1 \lessdot y_2 \lessdot \ldots \lessdot y_{l+1}$ and $z_1 \lessdot z_2 \lessdot \ldots \lessdot z_{l+1}$.

We first note that the unique lift of Proposition~\ref{prop lifting} is obtained by recursively applying Proposition \ref{P:uniquelift}.  Namely, the lift of a maximal face $F_1 < F_2 < \cdots < F_{l+1}$ is given by setting $u = u_1$ and setting $u_i = \min \{v \mid v \geq u_{i-1} \mbox{\ and\ } \pi(v) = F_i\}$.  This follows easily from the fact that the unique lift is a saturated chain in Bruhat order.

Thus for $i < a$, one has $x_i = \min \{v : v \geq x_{i-1} \mbox{\ and\ }
\pi(v)=M_i \}$ and $z_i = \min \{v : v \geq z_{i-1} \mbox{\ and\ } \pi(v)=M_i \}$ so we deduce by induction that $x_i=z_i$
for $i<a$. Now, we claim that $x_{i+1} \gtrdot z_i$ for $a-1 \leq i
< c$. Our proof is by induction on $i$; for $i = a-1$ we have $x_a >
x_{a-1} = z_{a-1}$, establishing the base case. For $i>a$, we have
$x_{i+1} = \min \{v : v \geq x_{i} \mbox{\ and\ } \pi(v)=M_i
\}$ and $z_i = \min \{v : v \geq z_{i-1} \mbox{\ and\ }
\pi(v)=M_i \}$ and our inductive hypothesis shows that the
right hand side of first equation is contained in the right hand
side of the induction, so $x_{i+1} \geq z_i$. But every link in the
chains $x_{\bullet}$ and $z_{\bullet}$ is a cover, so
$\ell(x_{i+1})=\ell(z_i)+1$ and we complete the induction.

The same arguments show that $y_i = z_i$ for $1 \leq i < b$.
Similarly, we have $y_i = x_i$ for $b < i \leq l+1$. So we have
$y_{b-1} = z_{b-1} \lessdot x_b \lessdot x_{b+1} = y_{b+1}$,
$y_{b-1} \lessdot y_b \lessdot y_{b+1}$ and $y_{b-1} = z_{b-1}
\lessdot z_b \lessdot x_{b+1} = y_{b+1}$ But there are only two
elements of $[u,w]$ between $y_{b-1}$ and $y_{b+1}$, so two of the
three of $x_b$, $y_b$ and $z_b$ are equal. We describe the case
where $x_b=y_b$; the other cases are similar. If $a < b$, then
$\pi(x_{b})=\pi(y_{b-1}) = M_{b-1}$. But also
$\pi(x_{b})=\pi(y_{b}) = \beta$, contradicting that
$\beta > M_{b-1}$. On the other hand, if $a=b$ then we have
$x_i=y_i$ for all $i$, so the two lifts $x_{\bullet}$ and
$y_{\bullet}$ are not actually different. We have now established
the first claim.

The second claim is very similar to the first.
We must either have $(u',w') = (u,w')$ with $w' \lessdot w$ or else $(u', w') = (u', w)$ with $u' \gtrdot u$;
we treat the former case.
If $\pi(w') \neq \pi(w)$, then the claim is easy:
$F$ does not contain $\pi(w)$ but every maximal face of $\pi(\Delta([u,w]))$ does, so the only maximal face of
$\pi(\Delta([u,w]))$ containing $F$ is the one whose additional vertex is $\pi(w)$.
Thus, we assume instead that $\pi(w') = \pi(w)$.

Suppose for contradiction that there are two faces of $\pi(\Delta([u,w]))$ containing $F$,
with additional vertices $\alpha$ and $\beta$ obeying $M_{a-1} < \alpha < M_a$, $M_{b-1} < \beta < M_b$
with $a \leq b$. Let $x_{\bullet}$ and $y_{\bullet}$ be the corresponding chains in $[u,w]$.
Also, let $z_{\bullet}$ be the chain in $[u,w']$ lifting $F$.
Then, as before, we show that $y_{b-1} \lessdot x_b \lessdot x_{b+1}=y_{b+1}$ and $y_{b-1} \lessdot y_b \lessdot y_{b+1}$.
Also, the same induction as before shows that $y_i=z_i$ for $i \leq b-1$ and $y_{i+1} \gtrdot z_i$ for $b \leq i \leq l$.
So $y_{b-1} = z_{b-1} \lessdot z_b \lessdot y_{b+1}$.
But there are only two elements of $[u,w]$ between $y_{b-1}$ and $y_{b+1}$ and we conclude as before.
\end{proof}

\section{Gr\"obner degeneration in the minuscule case} \label{sec:grobner}

For this section, suppose that $G/P$ is minuscule, so from the following list
(see e.g. \cite[chapter 9]{BL}):
\begin{itemize}
\item If $G = GL(n)$ or $SL(n)$: all ordinary Grassmannians are minuscule.
\item If $G = SO(N)$: the Grassmannian of orthogonal $\lfloor N/2\rfloor$-planes
  is minuscule. If $N$ is even, the quadric cone is minuscule.
\item If $G = Sp(n)$: projective space is minuscule.
\item If $G = E_6$: the Cayley plane is minuscule.
\item If $G = E_7$: one of the $G/P$ is minuscule.
\end{itemize}
Let $L$ be the minimal ample line bundle on $G/P$.
So all the weight spaces of $H^0(L, G/P)$ are one-dimensional, and are indexed by $W/W_P$.
Let $\I$ be an indexing set for these weight spaces. 

Let $k$ denote our ground field.  Let $k[p_I]$ be the polynomial ring whose variables are indexed by
$W/W_P$, so $G/P$ is naturally embedded in $\mathrm{Proj} \ k[p_I]$.
  For any simplicial complex $K$ on the vertex set
$W/W_P$, let $\SR(K)$ be the \defn{Stanley-Reisner ring} of
$K$; this is the quotient of $k[p_I]$ by the ideal generated by
all monomials which are not supported on $K$.  

Choose any total
order on $W/W_P$ refining the standard Bruhat order and let
$\omega$ be the corresponding reverse lexicographic term order on
the monomials of $k[p_I]$. We will refer to the $p_I$ as \defn{Pl\"ucker coordinates}, even though $G/P$ may not be a Grassmannian.
 If $J$ is a homogeneous ideal of
$k[p_I]$, let $\In_{\omega} J$ be the initial ideal of $J$ with
respect to $\omega$. If $X$ is a subvariety of projective space, let
$k[X]$ be the corresponding homogeneous coordinate ring.  We have
$k[X]= k[p_I]/J$ for a saturated homogeneous ideal $J = I(X)$
and we write $\In_{\omega} ( k[X] )$ for $k[p_I]/(\In_{\omega}
J)$.

It is well known \cite{SW} that $\In_{\omega} (k[G/P]) =
\SR(\Delta(W/W_P))$; this result traces back to Hodge.
The main result of this section is
the following generalization of this result:

\begin{theorem} \label{T:HodgePedoe}
  For any $u \leq w$,
  we have $\In_{\omega} (k[\Pi_{u}^w] ) = \SR(\pi(\Delta([u,w])))$.
\end{theorem}

We begin by establishing that $k[\Pi_{u}^w]$ has the correct Hilbert series.

\begin{prop}\label{P:sameHilbert}
  Let $d$ be any nonnegative integer. Then the degree $d$ summands of
  $k[\Pi_{u}^w]$ and $\SR(\pi(\Delta([u,w])))$ have the same dimension.
\end{prop}

\begin{proof}
  Recall that $\pi$ is the projection from $G/B$ to $G/P$.  By definition,
  $k[\Pi_{u}^w]_d$ is given by $H^0(\Pi_u^w, L^{\otimes d})=H^0( G/P, L^{\otimes d} \otimes
  \O_{\Pi_u^w})$. By Theorem~\ref{T:crepant} and \cite[Ex.~II.5.1(d)]{Har},
  this is isomorphic to $H^0(\pi^*(L^{\otimes d}), X_u^w)$.  We now use
  \cite[Theorem~34(ii)]{LL}. This states that the dimension of
  $H^0(X_u^w, \pi^*(L^{\otimes d}))$ is the number of ordered pairs $\big((M_1,
  M_2, \ldots, M_r)$, $(a_1, a_2, \ldots, a_r)\big)$ where $(M_1,
  M_2, \ldots, M_r)$ are the vertices of a face of
  $\pi(\Delta([u,w]))$ and the $a_i$ are rational numbers of the
  form $b/d$, for $b$ integral, such that $0 < a_1 < \cdots < a_r<
  1$. So each $(r-1)$-dimensional face of $\pi(\Delta([u,w]))$
  has $\binom{d-1}{r-1}$ choices for $(a_1, \ldots, a_r)$. There are
  also $\binom{d-1}{r-1}$ monomials of degree $d$ using exactly the variables
  of an $(r-1)$-dimensional face. So we have the desired equality.
\end{proof}

\begin{lemma} \label{L:GeomMonk}
  Let $u \leq w$ and let $p_{\pi(u)}$ be the Pl\"ucker coordinate
  on $G/P$ indexed by $\pi(u)$.
  Then set-theoretically, one has $X_u^w \cap \{ p_{\pi(u)} =0 \}
  = \bigcup_{u' \gtrdot_k u} X_{u'}^w \subseteq G/B$ and
  $\Pi_u^w \cap \{ p_{\pi(u)} =0 \}
  = \bigcup_{u' \gtrdot_k u} \Pi_{u'}^w \subseteq G/P$.
\end{lemma}

\begin{proof}
  If $w=w_0$, then the first claim says
  $X_u \cap \{ p_{\pi(u)} =0 \} = \bigcup_{u' \gtrdot_k u} X_{u'}$,
  which follows from the characterizations of Schubert varieties
  in \cite{FZrec}. Intersecting that with $X^w$ we get the general case.

  For the second part, let $H$
  denote the divisor $\{ p_{\pi(u)} =0 \}$ on $G/P$. So

\hfill
$\Pi_u^w \cap H = \pi(X_u^w) \cap H
  = \pi\left( X_u^w \cap \pi^{-1}(H) \right)
  = \pi \left( \bigcup_{u' \gtrdot_k u} X_{u'}^w \right)
  = \bigcup_{u' \gtrdot_k u} \Pi_{u'}^w.$ \hfill
\end{proof}

\newcommand\Proj{{\rm Proj\ }}
\newcommand\Slice{{\rm Slice}}
\newcommand\Cone{{\rm Cone}}

For $I$ an ideal of $k[x_1, \ldots, x_{n}]$, let
$\Slice(I) \leq k[x_1, \ldots, x_{n-1}]$
denote the image of the composite
$$ I \into  k[x_1, \ldots, x_{n}] \onto k[x_1, \ldots, x_{n}]/x_n
\iso k[x_1, \ldots, x_{n-1}]. $$
For $J$ an ideal of
$k[x_1, \ldots, x_{n-1}]$, view $k[x_1, \ldots, x_{n-1}]$ as a
subring of $k[x_1, \ldots, x_{n}]$
and let $\Cone(J) \leq k[x_1, \ldots, x_{n}]$  be the ideal
generated by $J$.

\begin{lemma} \label{L:revlex}
  Suppose $I$ is a homogeneous ideal in $k[x_1,...,x_n]$ such that
  $x_n f\in I \Rightarrow f \in I$. Then $\In(I) = \Cone(\In(\Slice(I)))$,
  where the initial ideals are defined with respect to reverse
  lexicographic (``revlex'') order.
\end{lemma}

\begin{proof}
Both sides of the equation are monomial ideals, so it is enough to
show that the minimal generators of each side are contained in the
other side. (A monomial $x^a$ is called a minimal generator of the
monomial ideal $J$ if $x^a \in J$ but no proper divisor of $x^a$
is in $J$.)

Let $x^a$ be a minimal generator of the LHS, so $x^a$ is the leading
term of $f$ for some homogeneous $f$. If $x_n$ divides $x^a$ then
$f/x_n$ is in $I$ and has leading term $x^a/x_n$, contradicting the
minimality of $x^a$. So $x_n$ does not divide $x^a$. Then $x^a$ is
also the leading term of $f|_{x_n=0}$ and hence lies in the RHS.

Conversely, suppose that $x^a$ is a minimal generator of the RHS.
Then $x^a$ is the leading term of $f|_{x_n=0}$ for some $f \in I$
and, without loss of generality, we may assume that $f$ is
homogeneous. Then $x^a$ is also the revlex-leading term of $f$, and hence
contained in the LHS.
\end{proof}

\begin{proof}[Proof of Theorem~\ref{T:HodgePedoe}.]
It is enough to show that $\In_\omega(I(\Pi_u^w)) \subseteq I(
\SR(\pi([u,w])))$, as Proposition \ref{P:sameHilbert} will then
imply that they are equal.

Our proof is by induction
on $\ell(w) - \ell(u)$; the base case where $w=u$ is obvious.  In
the following, the slices and cones are with respect to $x_n =
p_{\pi(u)}$.  (All the ideals contain the Pl\"ucker coordinates
$p_K$ for $K < \pi(u)$ so we shall ignore these coordinates.)

First we note that $p_{\pi(u)}$ is not a zero divisor in
$k[\Pi_u^w]$ because this ring is a domain and $\Pi_u^w$ is not
contained in $\{p_{\pi(u)} = 0\}$.  
By Lemma \ref{L:GeomMonk}, we have $\Slice(I(\Pi_u^w)) ) \subseteq
I\left( \cup_{u' \gtrdot_k u} \Pi_u^w \right)$.  Taking initial
ideals preserves containment, so $\In_\omega( I(
\cup_{u'\gtrdot_k u} \Pi_{u'}^w) ) \subseteq \bigcap_{u' \gtrdot_k
u} \In_\omega( I( \Pi_{u'}^w ) )$.  By induction, we have $\In(
I(\cup_{u'\gtrdot_k u} \Pi_{u'}^w )) \subseteq \bigcap_{u'\gtrdot_k
u} I( \SR( \pi([u',w])))$.  Combining this we get
$$
\Slice(I(\Pi(u^w))) \subseteq \bigcap_{u' \gtrdot_k u} I(\SR(
\pi([u',w]))).
$$
But by Lemma \ref{L:revlex}, we have
$$
\In_\omega(I(\Pi_u^w) ) = \Cone( \In_\omega( \Slice(I(\Pi_u^w))))
\subseteq \Cone\left( \bigcap_{u' \gtrdot_k u} I(\SR( \pi([u',w])))\right).
$$
Now, a face of $\bigcup_{u' \gtrdot_k u} \pi(\Delta([u',w]))$
is precisely the image under $\pi$ of a chain in $[u,w]$ whose
least element does not lie in $\pi^{-1}(u)$. In other words,
$\bigcup_{u' \gtrdot_k u} \pi(\Delta([u',w]))$ is precisely the
simplicial complex of all faces in $\pi(\Delta([u,w]))$ which
do not contain $\pi(u)$. Since every maximal face of
$\pi(\Delta([u,w]))$ contains $\pi(u)$, the cone on
$\bigcup_{u' \gtrdot_k u} \pi(\Delta([u',w]))$ is
$\pi(\Delta([u,w]))$.  Thus $\In_\omega(I(\Pi_u^w)) \subseteq
I( \SR(\pi([u,w])))$, as required.
\end{proof}


\section{From combinatorial to geometric properties}
\label{sec:combgeom}

Using the results of \S \ref{sec:shelling} and~\ref{sec:grobner},
we give alternate proofs of the geometric Corollaries \ref{C:Normal} and
\ref{c:CohenMacaulay} for minuscule $G/P$. While the argument establishing
Cohen-Macaulayness is standard (see, e.g. \cite{DL}), our criterion
establishing normality seems to be new even for Schubert varieties.

We emphasize that these are not truly independent proofs, as the results
of \S \ref{sec:grobner} relied on \cite{LL}, which itself
used Frobenius splitting. But there are other contexts where one has a
Gr\"obner degeneration to the Stanley-Reisner scheme of a ball
(e.g. \cite{GrobnerGeom}) where these arguments would apply.

\begin{prop}\label{prop:normalCM}
  Let $X = \coprod_E X_e$ be a projective variety with a
  stratification by normal (e.g. smooth) subvarieties.
  Assume that $X$ has a Gr\"obner
  degeneration to a projective Stanley-Reisner scheme $SR(\Delta)$.
  Then any subscheme of $X$ extends to a flat subfamily of the degeneration.
  Assume that
  \begin{enumerate}
  \item each $\overline{X_e}$ degenerates to $SR(\Delta_e)$,
    where $\Delta_e \subseteq \Delta$ is homeomorphic to a ball, and
  \item if $\overline{X_e} \supset X_f$, $e\neq f$,
    then $\Delta_f$ lies in the boundary $\partial \Delta_e$ of $\Delta_e$.
  \end{enumerate}
  Then each $\overline{X_e}$ is normal and Cohen-Macaulay.
\end{prop}

\begin{proof}
  Each $SR(\Delta_e)$ is Cohen-Macaulay \cite{Hochster}, and since Cohen-Macaulayness
  is an open condition, each $\overline{X_e}$ is also Cohen-Macaulay.

  Serre's criterion for normality is that each $\overline{X_e}$ be $S2$
  (implied by Cohen-Macaulayness) and regular in codimension $1$.
  If the latter condition does not hold on $\overline{X_e}$, then by the
  normality of $X_e$ the failure must be along some codimension $1$
  stratum $\overline{X_f} \subset \overline{X_e}$.

  However, by the assumption that $\Delta_f \subseteq \partial \Delta_e$
  and is of the same dimension,
  the scheme $SR(\Delta_e)$ is generically smooth along $SR(\Delta_f)$.
  Then by semicontinuity, $\overline{X_e}$
  is generically smooth along $\overline{X_f}$, contradiction.

  (It is amusing to note that while in topology one thinks of the
  boundary $\partial \Delta$ as the place where $\Delta$ is {\em not}
  a smooth manifold, in fact these are exactly the codimension $1$
  faces along which $SR(\Delta)$ {\em is} generically smooth.)
\end{proof}

\begin{cor}[Corollaries \ref{C:Normal} and \ref{c:CohenMacaulay}, redux]
  Projected Richardson varieties for minuscule $G/P$ are normal and Cohen-Macaulay.
\end{cor}

\begin{proof}
  We apply the Proposition above to the stratification of $G/P$
by open projected Richardson varieties. Condition (2) is
  the second conclusion of Proposition \ref{P:Thin}.
\end{proof}

\appendix
\section{Richardson varieties have rational resolutions in all characteristics}
\label{sec:appendix}
The purpose of this appendix is to establish the titular claim, and a bit more.  Brion \cite{Bripos} showed that Richardson varieties have rational resolutions in characteristic zero.
The extension to positive characteristic should be widely expected by experts, but we could not find it in the literature.  Most of the work has already been done by Michel Brion and Shrawan Kumar, and we thank them for suggestions as to how to complete the proof.

We first establish notation.  Let $p >0$ denote the characteristic.
If $(F_1,F_2) \in G\cdot (B,wB) \subseteq G/B \times G/B$, 
say that \defn{$F_1$ is $w$-related to $F_2$}. We will only need this
concept for $w$ an involution, in which case it is a symmetric relation.

Let $u \leq w$ be elements of $W$. Fix reduced words $s_{i_1} s_{i_2} \cdots s_{i_{\ell}}$ for $w$ and $s_{j_1} s_{j_2} \cdots s_{j_m}$ for $w_0 u$.
Let $r_1 r_2 \cdots r_{\ell+m}$ be the word $s_{i_1} \cdots s_{i_{\ell}} s_{j_m} \cdots s_{j_1}$. 
Let $Q \subset (G/B)^{\ell+m+1}$ be the subvariety of those sequences $(F_0, F_1, \ldots, F_{\ell+m})$ such that, for each $i$,  either $F_{i-1} = F_i$ or $F_{i-1}$ and $F_i$ are $r_i$-related. Let $D_i \subset Q$ be the hypersurface where $F_{i-1} = F_i$.

Let $p: Q \to (G/B) \times (G/B)$ be the projection onto $(F_0, F_{\ell + m })$. 
Let $Z = Z_u^w = p^{-1}(e B, w_0 B)$, let $\Delta_i = D_i \cap Z$ and let $\Delta$ be the divisor $\sum \Delta_i$.  Let $Z_u\to X_u$ and $Z^w \to X^w$ denote the Bott-Samelson varieties associated to the reduced words $s_{j_1} s_{j_2} \cdots s_{j_m}$  and $s_{i_1}s_{i_2} \cdots s_{i_\ell}$ above (see \cite[Chapter 2]{BK}, where the notation for Schubert varieties corresponds to our opposite Schubert varieties).  Then $Z_u^w = Z_u \times_{G/B} Z^w$.

Let $q: Z \to G/B$ be the projection onto $F_{\ell}$.  The following result follows from well known facts about Bott-Samelson varieties, and is implicit in the proof of \cite[Theorem 4.2.1]{Bri}.

\begin{lemma} \label{lem:ImageQ}
The image of $q$ is $X_u^w$. The map $q: Z \to X_u^w$ is birational, and $q$ gives an isomorphism $\Zo = Z \setminus \Delta \to \Xo_u^w$.
\end{lemma}
\begin{proof}
Let $\Zo_u \subset Z_u$ and $\Zo^w \subset Z^w$ denote the open subsets of the Bott-Samelson varieties, with the boundaries removed \cite{BK}.  Then $Z \setminus \Delta$ can be identified with $\Zo_u \times_{G/B} \Zo^w$.  But the maps $Z_u \to G/B$ and $Z^w \to G/B$ give isomorphisms $\Zo_u \simeq \Xo_u$ and $\Zo^w \simeq \Xo^w$, so we have $\Zo_u \times_{G/B} \Zo^w \simeq \Xo_u \cap \Xo^w = \Xo_u^w$.

It is clear that $q(Z)$ lies in $X_u^w$. 
Now, $Z$ is clearly closed in $(G/B)^{\ell+m+1}$, so it is proper and so the map $q: Z \to X_u^w$ is proper. 
Thus its image is closed and must contain the closure of $\Xo_u^w$; so the image contains $X_u^w$. Combining this with the previous paragraph, $q(Z) = X_u^w$.
\end{proof}

The following are the main theorems of this Appendix:

\begin{theorem} \label{thm:ZSmooth}
$Z$ is nonsingular.  The divisor $\Delta$ is anticanonical and $(p-1)\Delta$ induces a splitting of $Z$.
\end{theorem}

\begin{theorem} \label{thm:RatlRes}
We have $q_*(\O_Z) = \O_{X_u^w}$, $q_*(\omega_Z) = \omega_{X_u^w}$ and, for $j>0$, we have $R^j q_*(\O_Z) = R^j q_*(\omega_Z)=0$.
\end{theorem}

In the terminology of~\cite[Definition 3.4.1]{BK}, Theorem \ref{thm:RatlRes} says that $q: Z \to X_u^w$ is a rational resolution.  This is the characteristic $p$ version of rational singularities.

%

We now begin introducing the terminology we will use to prove Theorem~\ref{thm:ZSmooth}.
This Theorem, and in particular the smoothness of $Z$, is the result which we could not find a published proof of in arbitrary characteristic; all references use Kleiman transversality or related generic smoothness results which don't hold in finite characteristic.

Let $V \subset (G/B) \times (G/B)$ be the set of pairs of flags $(F, F')$ such that $F'$ is $w_0$-related both to $F$ and to $e B$.
Clearly, this is an open locus in $(G/B) \times (G/B)$. 
Define $\sigma: N_{+} \times N_{-} \to (G/B) \times (G/B)$ by $(n_{+}, n_{-}) \mapsto (n_{+} n_{-} B, n_{+} n_{-} w_0 B)$

\begin{lemma} \label{lem:UTrans}
The map $\sigma$ is an isomorphism between $N_{+} \times N_{-}$ and $V$.
\end{lemma}

\begin{proof}
We check bijectivity on points, and leave the rest to the reader.
First, we must show that the image of the map is $V$. 
Since $e B$ and $w_0 B$ are $w_0$-related to each other, so are their translations by $n_{+} n_{-}$. 
Also, we have $(B, n_{+} n_{-} w_0 B) = (n_{+} B, n_{+} w_0 B)$, so $B$ and $n_{+} n_{-} w_0 B$ are $w_0$-related.

Now, let $(F, F')$ be in $V$.
Since $F'$ is $w_0$-related to the standard flag, there is a unique $n_{+} \in N_{+}$ such that $n_{+} w_0 B = F'$. 
Then, for any $n_{-} \in N_{-}$, we will have $n_{+} n_{-} w_0 B = F'$.

The hypothesis that $F$ and $F'$ are $w_0$-related tells us that $n_{+}^{-1} F$ and $n_{+}^{-1}(F') = w_0 B$ are $w_0$-related.
So there is a unique $n_{-} \in N_{-}$ with $n_{-} B = n_{+}^{-1} F$.
In other words, there is a unique $n_{-} \in N_{-}$ with $n_{+} n_{-} B = F$. 

We have found the unique $(n_{+}, n_{-})$ such that $ (n_{+} n_{-} B, n_{+} n_{-} w_0 B)= (F,F')$.
\end{proof}

\begin{lemma} \label{lem:Prod}
We have $p^{-1}(V) \cong Z \times V$.
\end{lemma}

\begin{proof}
We describe maps $p^{-1}(V) \to Z \times V$ and $Z \times V \to p^{-1}(V)$; checking that they are inverse is straightforward. 
Given a point $F_{\bullet} = (F_0, F_1, \ldots, F_{\ell+m}) \in p^{-1}(V)$, let $(n_{+}, n_{-}) = \sigma^{-1}(p(F_{\bullet})))$; we map $F_{\bullet}$ to the point $(((n_+ n_-)^{-1} F_0, (n_+ n_-)^{-1} F_1, \ldots, (n_+ n_-)^{-1} F_{\ell+m}),p(F_{\bullet}))$ in $Z \times V$.
Conversely, given $((F_0, F_1, \ldots, F_{\ell+m}) , v) \in Z \times V$, let $(n_{+}, n_{-}) = \sigma^{-1}(v)$. We map  $((F_0, F_1, \ldots, F_{\ell+m}) , v) $ to $(n_+ n_- F_0, \ldots, n_+ n_- F_{\ell+m})$. \end{proof}



\begin{proof}[Proof of Theorem~\ref{thm:ZSmooth}]
First, suppose that $Z$ is singular. Then Lemma~\ref{lem:Prod} shows that $p^{-1}(V)$ is singular. But $p^{-1}(V)$ is open in $Q$ and $Q$, being a repeated $\mathbb{P}^1$ bundle over $G/B$, is smooth.  So we have a contradiction and $Z$ is smooth.

Next, we establish that $\Delta$ is anticanonical.  Let $A_{+} \subset (G/B)^{\ell+m+1}$ be the subvariety of sequences $(F_0, F_1, \ldots, F_{\ell}, \ldots, F_{\ell+m})$ so that, for $1 \leq i \leq \ell$, the pair $(F_{i-1}, F_i)$ are either equal or $r_i$ related and such that $F_0 = eB$; no condition is imposed on the $F_i$ for $i>\ell$.
Similarly, let $A_{-}  \subset (G/B)^{\ell+m+1}$ be the subvariety where $(F_{i-1}, F_i)$ are equal or $r_i$-related for $\ell+1 \leq i \leq \ell+m$.
Then $A_{+} \simeq Z^w \times (G/B)^{m}$ and $A_{-} \simeq (G/B)^{\ell} \times Z_u$ are products of Bott-Samelson varieties with many copies of $G/B$, so they are smooth; by definition, $Z = A_{+} \cap A_{-}$.
It is easy to compute the dimensions of $A_{\pm}$ and see that $Z$ is the transverse intersection of $A_{+}$ and $A_{-}$.

By Brion \cite[Lemma 1]{Bripos}, we have $\omega_{Z} \cong \omega_{A_+}|_Z \otimes  \omega_{A_-}|_Z \otimes \omega_{(G/B)^{m+\ell+1}}|_Z^{-1}$.
For any weight $\lambda$, let $L(\lambda)$ be the line bundle on $G/B$ which can be made $G$-equivariant such that the fiber over $eB$ transforms by $\lambda$.
For any index $i$ from $0$ to $\ell+m$,  let $L(\lambda, i)$ be the line bundle on $(G/B)^{m+\ell+1}$ which is pulled back from $L(\lambda)$ on $G/B$.
So the canonical bundle on $(G/B)^{\ell+m+1}$ is $\bigotimes_{i=0}^{\ell+m} L(- 2 \rho, i)$, where $2\rho$ is, as usual, the sum of the positive roots.
By \cite[Proposition 2.2.7]{Bri}, $\omega_{A_+} = \mathcal{O}(-\sum_{i=1}^{\ell} \Delta_i) \otimes L(-\rho, \ell) \otimes \bigotimes_{i=\ell+1}^{\ell+m} L(-2 \rho, i)$ and, similarly,  $\omega_{A_+} = \bigotimes_{i=0}^{\ell-1} L(-2 \rho, i)  \otimes L(-\rho, \ell) \otimes  \mathcal{O}(-\sum_{i=\ell+1}^{\ell+m} \Delta_i)$.
So $\omega_Z = \mathcal{O}(-\sum_{i=1}^{\ell+m} \Delta_i)$, as desired.

Finally, we must show that the anti-canonical divisor $\Delta$ induces a splitting.  Pick indices $i_{k_1},i_{k_2},\ldots,i_{k_r}$ and $j_{k'_1},j_{k'_2},\ldots,j_{k'_t}$ so that omitting these indices from $s_{i_1} \cdots s_{i_{\ell}} s_{j_m} \cdots s_{j_1}$ gives a reduced word for $w_0$.  Then the intersection $\Delta_{i_{k_1}} \cap \Delta_{i_{k_2}} \cap \cdots \Delta_{i_{k_r}} \cap \Delta_{j_{k'_1}} \cap \cdots \Delta_{j_{k'_t}}$ is a single point $(F_0, F_1, \ldots, F_{\ell+m})$ where each flag $F_i = w_i B$ (for some $w_i \in W$) is torus-invariant.  Furthermore, it is clear that this is intersection is transverse, and does not intersect any other $\Delta_j$: there is no way to change $B$ to $w_0B$ in less than $\ell(w_0)$ steps.  By \cite[Proposition 1.3.11]{BK}, the divisor $(p-1)\Delta$ induces a splitting of $Z$.
\end{proof}

\begin{lemma}\label{AmpleDi}
There are nonnegative integers $a_i$ such that $\mathcal{O}(\sum a_i \Delta_i)$ is ample on $Z$.
\end{lemma}

\begin{proof}
Let $L$ be a very ample line bundle on $G/B$, and for $0 \leq i \leq \ell+m$, let $\pi_i: Z \to G/B$ denote the projection.  Then clearly $\bigotimes_{i=0}^{\ell+m} \pi_i^*(L)$ is an ample line bundle on $Z$.  It is enough to show that each $\pi_i^*(L)$ is isomorphic to $\mathcal{O}(\sum a_j \Delta_j)$ for some nonnegative integers $a_j$.  For each $i$, the map $\pi_i: Z \to G/B$ factors into the composition $\pi = \psi \circ \phi$, of $\phi: Z \to Z'$ and $\psi: Z' \to G/B$ where $Z' =Z_x$ or $Z^y$ is a Bott-Samelson variety.  The pullback $\psi^*L$ is an effective line bundle on $Z'$, so by \cite[Exercise 3.1.E.3(e)]{BK} is isomorphic to $\mathcal{O}_{Z'}(\sum a'_k \Delta'_k)$ for nonnegative integers $a'_k$, where $\Delta'_k$ are the boundary divisors of $Z'$, defined in a similar manner to the $\Delta_j$.  But the inverse image $\phi^{-1}(\Delta'_k)$ in $Z$ is some $\Delta_j \subset Z$ so $\phi^*(\mathcal{O}_{Z'}(\Delta'_k)) \simeq \mathcal{O}_Z(a \Delta_j)$ for some positive $a$ and some $j$.  It follows that $\pi_i^*(L)$ is isomorphic to $\mathcal{O}(\sum a_j \Delta_j)$ for some nonnegative integers $a_j$.
\end{proof}

Our proof is an adaptation of the proof of \cite[Theorem 3.1.4]{BK}. 

\begin{proof}[Proof of Theorem~\ref{thm:RatlRes}]
From Lemma~\ref{lem:ImageQ}, the map $q$ is an isomorphism on $Z \setminus \Delta$, so $\Delta$ contains the exceptional locus of $q$.
So~\cite[Theorem~1.3.14]{BK} applies, and we deduce that $R^j q_* \omega_Z = 0$ for $j>0$.

Next, $X_u^w$ is normal (the argument in \cite{Bripos} holds in all characteristics) and the map $q: Z \to X_u^w$ is birational and proper, so $q_* \mathcal{O}_Z = \mathcal{O}_{X_u^w}$. 

Now, let $L$ be a very ample line bundle on $X_u^w$. 
Let $a_i$ be as in Lemma~\ref{AmpleDi}.
Choose $\nu$ large enough that $p^\nu > a_i$ for all $i$.
Let $j>0$.
As $L$ is very ample, $q^* L$ is globally generated. By construction, $ \mathcal{O}(\sum a_j D_j)$ is ample, so $(q^* L)^{p^{\nu}} \otimes \mathcal{O}(\sum a_j D_j)$ is ample.
Since $Z$ is split, this implies that $H^i(Z, (q^* L)^{p^{\nu}} \otimes \mathcal{O}(\sum a_j D_j))=0$ (\cite[Theorem~1.2.8]{BK}).
But by~\cite[Lemma~1.4.11]{BK}, $H^i(Z, q^* L) $ injects into $H^i(Z, (q^* L)^{p^{\nu}} \otimes \mathcal{O}(\sum a_j D_j))$. So $H^j(Z, q^* L)=0$. 
By the same argument,  $H^j(Z, q^* L^n)=0$ for all positive $n$.
Then, by \cite[Lemma 3.3.3]{BK}, we deduce that $R^j q^* \mathcal{O}_Z=0$.

We have checked all the conditions except that $q_* \omega_Z = \omega_{X_u^w}$. By~\cite[Lemma 3.4.2]{BK}, this follows from the others.
\end{proof}

\end{document}